\documentclass[a4paper,11pt]{amsart}
\pdfoutput=1 
\usepackage{amsmath,amsthm,amssymb}
\usepackage[mathscr]{eucal}
 \usepackage{cite}
\usepackage{upgreek}
\usepackage[bookmarks=false]{hyperref}
\usepackage{enumerate}
\usepackage{amsmath}
\usepackage{mathrsfs}
\usepackage{blindtext}
\usepackage{scrextend}
\usepackage{enumitem}
\usepackage{bm} 
\addtokomafont{labelinglabel}{\sffamily}
\usepackage{color}
\usepackage[at]{easylist}
\usepackage{tikz}
\usepackage{svg}
\usepackage{graphicx}

\usepackage{comment}

\numberwithin{equation}{section}

\newtheorem{main}{Theorem}

\newtheorem{thm}{Theorem}[section]
\newtheorem*{thm*}{Theorem}
\newtheorem{lem}[thm]{Lemma}
\newtheorem*{prob*}{Problem}
\newtheorem{fact}[thm]{Fact}

\newtheorem{prop}[thm]{Proposition}
\newtheorem*{prop*}{Proposition}

\newtheorem{cor}[thm]{Corollary}
\newtheorem*{cor*}{Corollary}

\theoremstyle{definition}
\newtheorem{defn}[thm]{Definition}
\newtheorem*{defn*}{Definition}

\newtheorem{remark}[thm]{Remark}

\newtheorem*{question*}{Question}
\newtheorem*{Pquestion*}{Popa's question}

\newtheorem*{conv*}{Convention}

\newcommand{\N}{\mathbb{N}}

\newcommand{\C}{\mathbb{C}}
\newcommand{\Z}{\mathbb{Z}}

\newcommand{\F}{\mathbb{F}}

\newcommand{\cH}{\mathcal{H}}

\newcommand{\cO}{\mathcal{O}}

\newcommand{\cS}{\mathcal{S}}

\newcommand{\cU}{\mathcal{U}}
\newcommand{\cV}{\mathcal{V}}

\newcommand{\fO}{\mathfrak{O}}
\newcommand{\fR}{\mathfrak{R}}

\newcommand{\ee}{\varepsilon}

\newcommand{\Img}{\operatorname{Im}}
\newcommand{\Real}{\operatorname{Re}}

\begin{document}

\title[ sequential commutation in tracial von Neumann algebras]
{ sequential commutation in tracial von Neumann algebras}

\author[Srivatsav Kunnawalkam Elayavalli]{Srivatsav Kunnawalkam Elayavalli}
\address{Department of Mathematical Sciences, UCSD, 9500 Gilman Dr, La Jolla, CA 92092, USA}\email{srivatsav.kunnawalkam.elayavalli@vanderbilt.edu}
\urladdr{https://sites.google.com/view/srivatsavke/home}

\author[Gregory Patchell]{Gregory Patchell}
\address{Department of Mathematical Sciences, UCSD, 9500 Gilman Dr, La Jolla, CA 92092, USA}\email{gpatchel@ucsd.edu}

\begin{abstract}

Recall that a unitary in a tracial von Neumann algebra is Haar if $\tau(u^n)=0$ for all $n\in \mathbb{N}$. We introduce and study a new Borel equivalence relation $\sim_N$ on  the set of  Haar unitaries in a diffuse tracial von Neumann algebra $N$. Two Haar unitaries $u,v$ in $\mathcal{U}(N)$ are related if there exists a finite path of sequentially commuting Haar unitaries in an ultrapower $N^\mathcal{U}$, beginning at $u$ and ending at $v$. We show that for any diffuse tracial von Neumann algebra $N$, the equivalence relation $\sim_N$ admits  either 1 orbit or uncountably many orbits. We characterize property Gamma in terms of path length and number of orbits of $\sim_N$ and also show the existence of non-Gamma II$_1$ factors so that $\sim_N$ admits only 1 orbit.   Examples where  $\sim_N$ admits uncountably many orbits include $N$ having positive 1-bounded entropy: $h(N)>0$. As a key example, we explicitly  describe $\sim_{L(\mathbb{F}_t)}$ for the free group factors. Using  these ideas we introduce a natural numerical invariant for diffuse tracial von Neumann algebras called the commutation diameter, with applications to elementary equivalence classification. This computes the largest value of the minimal path length over related unitaries.  We show that if $N$  admits one orbit, then the commutation diameter  is finite and moreover is an elementary equivalence invariant. By studying  the finer technical structure of lifts of certain commutators in ultrapowers  we   obtain non-trivial lower bounds for the family of arbitrary graph products $N$ of diffuse tracial von Neumann   algebras whose underlying graph is   connected and has diameter at least 4, and distinguish  them up to elementary equivalence from the  \cite{exoticCIKE} exotic factors, despite satisfying $h(N)\leq 0$.

\end{abstract}
\maketitle

\begin{center}
    \emph{\small{Dedicated to Marc A. Rieffel}} 
\end{center}

\section{Introduction}

The study of the structure of tracial von Neumann algebras by leveraging algebraic tools such as commutation has been central in the development of the subject.  Despite having no center, II$_1$ factors could have approximate central  sequences, i.e., property Gamma of Murray and von Neumann \cite{MuvN43}. The study of the structure of the central sequence algebra has been the source of many insights in the subject, with relation to classification and non-isomorphism results (\cite{McD1, McD2, Connes}), and more recently pertaining to  the new considerations in the last  decade involving model theory and structure of ultrapowers \cite{FHS, GH, BCI15}. Note that two II$_1$ factors $M,N$ are {\it elementarily equivalent} if and only if they admit isomorphic ultrapowers, $M^\mathcal U\cong N^\mathcal V$, for some ultrafilters $\mathcal U,\mathcal V$ on arbitrary sets \cite{FHS,hensoniovino}. Ideas involving commutation tend to naturally appear in the coarser problem of classifying II$_1$ factors up to elementary equivalence. More recently a new structural property involving commutation without the existence of central sequences was discovered in \cite{exoticCIKE}, with various applications. In the present article we are inspired by some of these ideas and subsequent  developments in \cite{houdayer2023asymptotic}. 

We build in this paper a new systematic viewpoint on the structure theory of tracial von Neumann algebra based on the idea of sequential commutation. One says that a sequence of elements $x_1,\hdots, x_n$ sequentially commute if $x_i$ commutes with $x_{i+1}$ for each $1\leq i\leq n-1$. This already imposes interesting analytic rigidity for the object these elements generate; for instance, on the level of countable groups, if $g_i$ are not torsion, this implies cost 1 (see \cite{gaboriau2000cout}) as well as vanishing of the first $\ell^2$-Betti number (see \cite{PetersonThom}, also for an explicit argument see Proposition 5.3 of \cite{brothier2022forestskein}).  In the von Neumann algebra setting, this has been considered at various points. Popa has considered a property resembling  sequential commutation in his Property C (\cite{Popa83}) and C' with Galatan (\cite{GalatanPopa}) related to the study of cohomology theory for von Neumann algebras (see also \cite{DykemaFreeEntropy, Hayes2018}). Voiculescu made a key computation of the free entropy dimension under a sequential commutation assumption \cite{VoiculescuPropT} (see also \cite{GeShenPropT}). Ge  and  Shen  in \cite{GeShenGen} proved single generation of certain von Neumann algebras generated by sequentially commuting unitaries, for instance $L(SL_3(\mathbb{Z}))$. Hayes in \cite{Hayes2018} 
 (see also \cite{JungSB}) showed that the family of diffuse tracial von Neumann algebras $N$ satisfying $h(N)>0$ cannot  be generated by sequentially commuting diffuse  unitaries. 

Recently in \cite{exoticCIKE} the authors constructed a non-Gamma factor $N$ satisfying the condition that any pair of independent unitaries (see Definition \ref{defn-indep}) in a dense family admits two sequentially commuting Haar unitaries between them. This was used in connection with a technical lifting lemma to derive applications to the structure of the ultrapowers of these factors. More recently \cite{houdayer2023asymptotic} proved using non-commutative $L^p$-space theory \cite{mei} the lack of a  sequentially commuting extension of size $2$ in certain pairs of elements  in free products. This was used again in conjunction with their new lifting lemma to show that  the factor $N$ in \cite{exoticCIKE} is not elementarily equivalent to $N*\mathbb{Z}$.     Having motivated the consideration, we are presently interested in developing a systematic  understanding  of the theory of sequential commutation in tracial diffuse von Neumann algebras, with new applications to considerations related to the structure theory of ultrapowers. The more appealing setup is to work inside the more flexible  setting  of the tracial ultrapower rather than just the ambient tracial diffuse von Neumann algebras, because it captures approximate sequential commutation.

Let $(N,\tau)$ be a diffuse tracial von Neumann algebra. Recall that a  unitary $u\in \mathcal{U}(N)$ is said to be Haar if $\tau(u^n)=0$ for all $n\in \mathbb{N}$. Let $\cU$ be a countably cofinal ultrafilter on a  set $I.$ We introduce and study the following equivalence relation on Haar unitaries $\mathcal{H}(N)$, denoted by $\sim_N$,  given by $u\sim_N v$ if there are $w_1,\ldots,w_n \in \mathcal{H}(N^\mathcal{U})$ such that $[u,w_1] =[w_k,w_{k+1}] = [w_n,v] = 0$ for all $1\leq k < n$. First we establish the following formalism of this equivalence relation.    
For any diffuse tracial von Neumann algebra $(N,\tau)$, we have that $\sim_N$ is a Borel equivalence relation, when endowing $\mathcal U(N)$ with the strong operator topology and then considering the restriction to $\mathcal H(N)$. The relation $\sim_N$ is in particular $F_\sigma$. Moreover $\sim_N$ does not depend on choice of the  ultrafilter. In connection with the above we also prove a useful lifting lemma of independent interest (see Lemma \ref{haar lifting OP}) for Haar unitaries in ultrapowers.

The first natural point of investigation is to characterize the existence of  genuine central sequences in terms of the above equivalence relation. To do this we need the notion of path lengths.   Let $u,v \in \cH(N).$ Then $p_N(u,v)$ denotes the length of the shortest possible sequentially commuting path between $u$ and $v$ in $N^{\mathcal{U}}.$ We define $p_N(u,u) = 0.$ If $[u,v]=0$ and $u\neq v$ then $p_N(u,v) = 1.$ If there is $w\in \cH(N^{\mathcal{U}})$ such that $[u,w]=[w,v]=0$ then $p_N(u,v)\leq 2.$ If $u\not\sim_{N} v$ then we set  $p_N(u,v) = \infty.$ We prove the appropriate characterization of property Gamma: $\sim_N$ admits a unique orbit, and the upper bound of 2 on the shortest path lengths between each pair of Haar unitaries. Surprisingly, on the contrary we find that  there are  many non-Gamma factors $N$  whose $\sim_N$ admits a unique orbit and interestingly feature a uniform upper bound on the shortest path lengths. For instance, there exist non-Gamma factors $N$ such that   for all $u_1,u_2\in \cH(N),$ $p_N(u_1,u_2) \leq 6.$

Next we obtain examples of II$_1$ factors $N$ such that $\sim_N$ admits uncountably many orbits.  Our examples come from 1-bounded entropy theory (\cite{JungSB, Hayes2018}). In particular, we use the Baire Category Theorem in combination with the observation that the 1-bounded entropy of the von  Neumann algebra generated by an orbit has to have vanishing 1-bounded entropy.  Using the recent resolution of the Peterson-Thom conjecture \cite{bordenave2023norm, PTkilled, HayesPT, HJKEPT}, we are also able to make the following explicit  computation which demonstrates the reach of the 1-bounded entropy methods for free group factors.   If $N$ is a II$_1$ factor satisfying $h(N)>0$, then $\sim_N$ admits uncountably many  orbits. If $N$ is moreover separable, then by Silver's dichotomy $\sim_N$ has exactly $\mathfrak c$ many orbits. In particular, if $M$ is a non-amenable II$_1$ factor such that $\sim_M$ admits only one orbit, then $M$ does not   embed into $L(\mathbb F_2)$. Moreover, let $u,v\in \mathcal{H}(L(\mathbb{F}_t))$ for $t>1$ or $t=\infty$. Then $u\sim v$ if and only if $\{u,v\}''$ is amenable. This result reveals a new intrinsic description of Pinsker algebras in free group factors (see \cite{FreePinsker, HayesPT}), which is of independent interest.

Our main goal for the paper is to introduce quite natural invariants for diffuse tracial von Neumann algebras that provide new insights for structure theory of ultrapowers.  The first notion should be understood as a ``commutation diameter,'' analogous and in a mild way inspired by the definition of diameter of a graph. 
    Let $N$ be a diffuse tracial von Neumann algebra. Denote the commutation diameter of $N$ by $\mathfrak{R}(N)= \sup_{u\sim_N v} p_N(u,v)$. We also need the following notion: Let $\mathfrak{O}(N)$  denote the number of orbits of $\sim_N$. $N$ is abelian, tracial, and diffuse if and only if $ \mathfrak{O}(N)= 1$ and $\mathfrak{R}(N) = 1;$ for II$_1$ factors, $N$  has property Gamma if and only if $\mathfrak{O}(N)= 1$ and $\mathfrak{R}(N) = 2$. For free group factors, $ \mathfrak{O}(L(\F_t))= \mathfrak{c}$  and $\mathfrak{R}(L(\F_t)) = 2$. From Theorem \ref{cike-one-orb} we have there exists a non-Gamma II$_1$ factor $N$ such that $\fO(N) = 1$ and $3 \leq \mathfrak{R}(N) \leq 6$.  For any diffuse tracial von Neumann algebra $(M,\tau)$ and ultrafilter $\mathcal{U}$, we also have     $\mathfrak{R}(M) \leq \mathfrak{R}(M^\cU)$ and $\mathfrak{O}(M) \leq \mathfrak{O}(M^\cU)$.  Moreover we show  the following notable results that reveal a new useful invariant for ultrapowers: 
\begin{main}\label{mainprop list}
For any diffuse tracial von Neumann algebra $(M,\tau)$, $\fO(M) = 1$ or $\fO(M) \geq \aleph_1.$ Moreover, if $\fO(M) = 1$, then $\fR(M) < \infty$ and $\mathfrak{R}(M)=\mathfrak{R}(M^\mathcal{U})$, i.e, the  commutation diameter is an invariant of elementary equivalence. 
\end{main}

We obtain in this paper an application to elementary equivalence classification (for some results in this project see \cite{FGL, FHS, GH, BCI15, AGKE, exoticCIKE, goldbring2023uniformly, houdayer2023asymptotic}). By developing new technical insights on the structure of lifts of commutators in ultrapowers (an instance of which is an  asymptotic orthogonality phenomenon in graph products), we are able to prove the result below, which demonstrates the usefulness of our theory of sequential commutation. 

\begin{main}\label{artin}
        Let $N$ be a  graph product of diffuse tracial von Neumann  algebras where the underlying graph is   connected and has diameter at least 4. Then $\mathfrak{R}(N)\geq 4$ and moreover $N$ is not elementarily equivalent to the \cite{exoticCIKE} non-Gamma II$_1$ factors. 
\end{main}

Recall that the notion of a graph product is a common generalization of free and tensor products introduced by \cite{Gr90} for groups and then in \cite{CaFi17} for von Neumann algebras. The {diameter}                     of a graph $\Gamma$ is given by $\sup_{v,w\in V(\Gamma)} d(v,w).$ In \cite{charlesworth2023strong2} it is shown that $N$ as in Theorem \ref{artin} is non-Gamma; moreover, it is shown in \cite{charlesworth2023strong} that $h(N)\leq 0$; therefore, the above result is new and not covered by existing parts of Theorem A in \cite{exoticCIKE} and Theorem F in  \cite{houdayer2023asymptotic}. An illustrative example of such an $N$ is the group von Neumann algebra of the right angled Artin group whose underlying graph is a line segment with 5 points.  

We end  with a natural line of investigation that is opened up by this work, potentially leading to the construction of a countable infinity of  concrete pairwise elementarily inequivalent non-Gamma factors. To do this we define a modified \cite{exoticCIKE} construction on $N$ by gluing handles of sequentially commuting  $n$ paths of Haars on pairs of independent unitaries (see Section \ref{CIKE construction}), and call it $\mathcal{S}_n(N)$. We conjecture that
 $\mathfrak{R}(\mathcal{S}_n(\mathbb{F}_2))\geq n+1$; in particular, $\mathcal{S}_{2^n - 2}(\mathbb{F}_2)$ gives an infinite family of mutually non-elementary equivalent non-Gamma II$_1$ factors. 
The proof of the in particular part uses our computation that $\mathfrak{R}(\mathcal{S}_n(\mathbb{F}_2))\leq 2(n+1)$ (see Proposition \ref{upperbound}). There are two conceptual difficulties in establishing the above conjecture. Firstly, in the process of proving the first step of non-Gamma, we crucially required in \cite{exoticCIKE} the use of the input $N$ to have Property (T), which we do not have in this situation. Secondly, at the moment, we are unable to compute any lower bounds beyond $4$.  However, we now understand more of the global structure of the free group factors (\cite{HJKEPT}), which we believe will help in proving the above conjecture. In fact the main motivation for the conjecture comes from Theorem \ref{thm-FGF} in particular.

\subsection*{Structure of the paper}
In Section 2, we recall some basic facts and important theorems from the literature that we will use in this paper. In Section 3 we define $\sim_N$, establish the basic formalisms and prove a useful ultrapower density lemma. Section 4 is where we relate Property Gamma to $\mathfrak{O}(N)$ and $\fR(N)$. Section 5 shows the existence of II$_1$ factors with uncountably many orbits,  and fully clarifies the particular case of the free group factors. The paper culminates in Section 6 wherein we introduce and prove our main results about the commutation diameter,  and Theorem \ref{artin}. 
\subsection*{Acknowledgements} We thank David Sherman for a motivating conversation in the Spring of 2023 that in hindsight underpins the formation of this paper.   We thank Adrian Ioana and Forte Shinko for helpful discussions, Adrian Ioana and Cyril Houdayer for helpful comments and Srivatsa Srinivas for help with generating figures and for comments around the bound in Lemma \ref{lemma-lower-bound}. We warmly thank Jesse Peterson for sending us several insightful remarks on a previous draft. We thank an anonymous referee for helpful suggestions that improved the exposition.

\section{Preliminaries}

\subsection{Background and notation on relevant aspects of tracial von Neumann algebras}
Let $(M,\tau)$ be a tracial von Neumann algebra, i.e., a pair consisting of a von Neumann algebra $M$ and a faithful normal tracial state $\tau:M\rightarrow\mathbb C$. If in addition $M$ is infinite dimensional and $\mathcal{Z}(M) := M'\cap M= \mathbb{C}$, then $M$ is II$_1$ factor.  We denote by $\mathcal U(M)$ the group of unitaries in $M$ and by $\mathcal{H}(M)$ the set of Haar unitaries in $M$, i.e., all unitaries such that $\tau(u^n) = 0$ for all $n\in \mathbb{N}$. Observe that any unitary conjugate of a Haar unitary is again Haar. Observe also that if $(M,\tau)$ is a separable diffuse tracial von Neumann algebra  then $M$ is the SOT-closed span of $\mathcal{H}(M)$ and $1$; indeed, any self-adjoint $a\in M$ lives in a diffuse abelian von Neumann subalgebra which is isomorphic to $L(\mathbb{Z})$ whose standard generator is a Haar unitary.

Denote also $M_{\text{sa}}$ the set of self-adjoint elements of $M$. Given a self-adjoint set $S\subset M$, von Neumann's bicommutant theorem implies that $S''$ is the smallest unital von Neumann subalgebra of $M$ containing $S$.

Let $\mathcal{U}$ be an ultrafilter on a set $I$, we denote by $M^{\mathcal U}$ the \emph{tracial ultraproduct} i.e, the quotient $\ell^\infty(I,M)/\mathcal{J}$ by the closed ideal $\mathcal{J}\subset\ell^\infty(I,M)$  consisting of $x=(x_n)_{n\in I}$ with $\lim\limits_{n\rightarrow\mathcal U}\|x_n\|_2= 0$.  $M$ admits a natural diagonal inclusion $\iota:M\to M^\cU$ given by $\iota(x) = (x)_{n\in I}$. For notational simplicity we identify $M$ with $\iota(M).$

We recall the following folklore result about embeddings into ultrapowers which is an immediate corollary of Proposition 2.6.4 of \cite{thesis}.

\begin{prop}\label{jung-atkinson}
    Let $N$ be a II$_1$ factor with trace $\tau$ and $M$ a separable hyperfinite tracial von Neumann algebra. Let $\cU$ be a free ultrafilter on $\N.$ If $\pi,\rho:M\to N^\cU$ are trace-preserving embeddings, then $\pi,\rho$ are unitarily conjugate.

    In particular, if $u,v\in \cH(N^\cU)$ are Haar unitaries then there is a unitary $w\in \cU(N^\cU)$ such that $v = wuw^*$.
\end{prop}

A II$_1$ factor $M$ has {\it property Gamma} if  $M'\cap M^{\mathcal U}\neq \mathbb{C}1$, for a free  ultrafilter $\mathcal U$ on $\mathbb N$. The following is an equivalent characterization of property Gamma.

\begin{thm}[see Proposition 3.8 of \cite{shen2019reducible}]\label{shenshi}
    Let $M$ be a II$_1$ factor. If for all $x\in M$ we have that $W^*(x)'\cap M^\cU\neq \C1$, then $M$ has property Gamma. 
\end{thm}

\subsection{A lifting lemma} The following Haar lifting lemma is a particular case of Lemma 1.3 
in \cite{popa1987commutant}, but we give here a proof for the reader’s convenience. We thank Adrian Ioana for helpful discussions around the proof.

\begin{lem}\label{haar lifting OP}
    Let $(M,\tau)$ be a diffuse tracial von Neumann algebra and let $D\subset \mathcal{H}(M)$ be an SOT-dense subset. Let $\cU$ be a free ultrafilter on $\N.$ Then for all $u\in \mathcal{H}(M^\cU)$, there are $u_n\in D$ such that $u = (u_n)_n.$
\end{lem}

\vspace{-10pt}

\begin{proof}
    It suffices to show that if $u\in \cH(M^\cU)$ can be written as $u = (u_n)_n$, then for all $\ee>0$ there is $A\in\cU$ such that for all $n\in A,$ $\|u_n - v\|_2 < \ee $ for some $v\in \cH(M).$

    Fix $\ee>0$. Pick a natural number $N$ such that $\frac{4\pi+\sqrt{6}}{N} < \ee.$

    Set $E = \{e^{2\pi i \theta} : 0\leq \theta < 1/N\}$ as in the left image of Figure \ref{fig:enter-label}. Now, there is a continuous functions $f : \mathbb T \to [0,1]$ such that the set $\{x : f(x) \neq \chi_E(x)\}$ is contained in the disjoint union of two arcs of $\mathbb T,$ each of which has Lebesgue measure $\frac{1}{N^5}$, each of which are centered at the endpoints of $E$, as in the middle image of Figure \ref{fig:enter-label}. Define $S = \{e^{2\pi i\theta } : -\frac{1}{2N^5} < \theta < \frac{1}{2N^5}\}$. Then $\{x : f(x) \neq \chi_E(x)\} = S \sqcup e^{2\pi i \frac{1}{N}} S$. Define $T = \sqcup_{k=1}^N e^{2\pi i \frac{k}{N}}S$ as in the right image of Figure \ref{fig:enter-label}. Define $T_k = e^{2\pi i \frac{k}{N^4}} T.$ Then $\sqcup_{k=1}^{N^4} T_k \subset \mathbb T$.

    \begin{figure}
        \centering
        \includegraphics[width=6.5in,trim={0 2.5in 0 1.8in},clip]{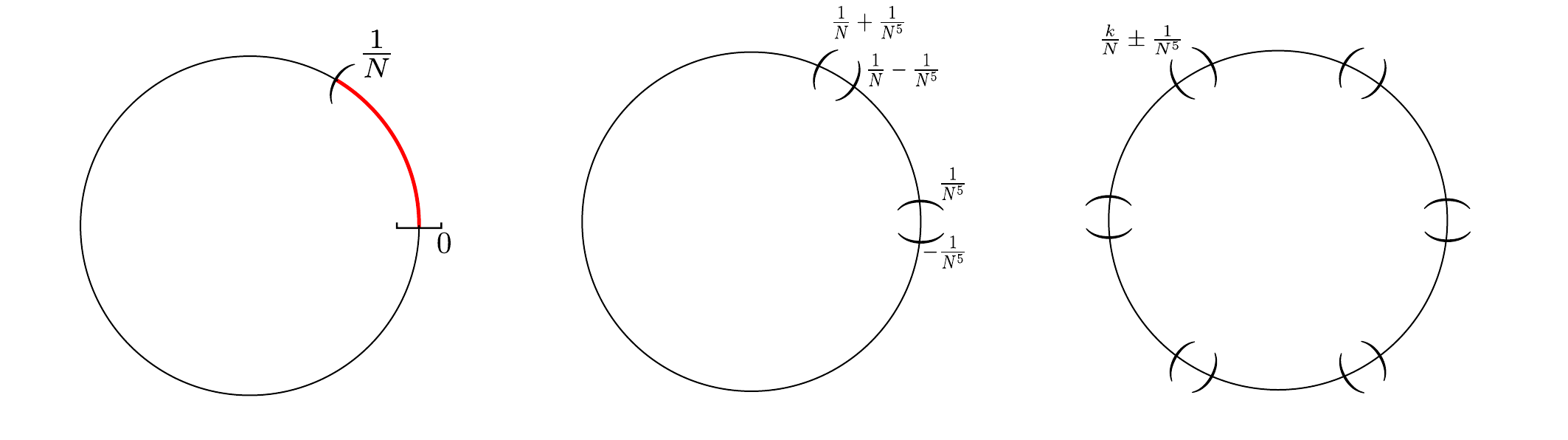}
        \caption{The sets $E,$ $\{x : f(x) \neq \chi_E(x)\}$, and $T$ respectively.}
        \label{fig:enter-label}
    \end{figure}

    By the Stone-Weierstrass Theorem, there is a polynomial $p$ (in $x$ and $\overline{x}$) such that $p:\mathbb T \to \mathbb R$ and $\|p-f\|_\infty < \frac{1}{N^4}$. Let us write $p(x) = \sum_{j=-M}^M a_jx^j$. Set $c = \sum_{j=-M}^M |a_j|$. By the ultrapower construction, there is a set $A\in\cU$ such that for all $n\in A,$ $|\tau(u_n^m)| < \frac{1}{cN^4}$ for all $0<|m| \leq M$.

    So fix $n\in A$. Recall that we have a unitary $u_n$ coming from the lift of $u$; $u_n$ gives us a spectral measure $\mu$ on $\mathbb T$. There must be some $1\leq k\leq N^4$ such that $\mu(e^{2\pi i\frac{k}{N^4}} T) \leq \frac{1}{N^4}$. Set $\alpha = e^{2\pi i\frac{k}{N^4}}.$ Define $\alpha_j = e^{2\pi i \frac{j}{N}}\alpha$ for all $1\leq j\leq N$. Note that $\alpha_j T =\alpha T$ for all such $j.$

    We now compute that for all $1\leq j\leq N,$
    \begin{align*}
        \tau(\chi_E(\alpha_j^{-1}u_n) - f(\alpha_j^{-1}u_n)) &= \int_\mathbb T \chi_E(\alpha_j^{-1}x) - f(\alpha_j^{-1}x) d\mu(x) \\
        &\leq \mu(\{x : f(\alpha_j^{-1}x) \neq \chi_E(\alpha_j^{-1}x)\})\\
        &\leq \mu(\alpha_j T)\\
        &= \mu(\alpha T)\\
        &\leq \frac{1}{N^4}.
    \end{align*}

    We also have that $|\tau(f(\alpha_j^{-1}u_n) - p(\alpha_j^{-1}u_n))| \leq \|f-p\|_\infty < \frac{1}{N^4}$.

    Hence $|\tau(p(\alpha_j^{-1}u_n)) - \tau(\chi_E(\alpha_j^{-1}u_n))| < \frac{2}{N^4}. $
    
    By our choice of $A\in\cU$, we also have that $|\tau(p(u_n)) - p(0)| < \frac{1}{N^4}.$ Therefore we also have that $|\tau(p(\alpha_j^{-1}u_n)) - p(0)| < \frac{1}{N^4}$ for all $1\leq j\leq N$.

    Hence $|p(0) - \tau(\chi_E(\alpha_j^{-1}u_n))| < \frac{3}{N^4}$ for all $j$. Now note that $\chi_E(\alpha_j^{-1}u_n) = \chi_{\alpha_j E}(u_n)$ and that the disjoint union of the $\alpha_j E$ is $\mathbb T.$ Therefore $\sum_{j=1}^N \tau(\chi_{\alpha_j E}(u_n)) = 1$. Since each of the summands is at most $\frac{6}{N^4}$ away from each other summand, we have that $|\tau(\chi_{\alpha_j E}(u_n)) -\frac{1}{N}| < \frac{6}{N^4} $ for all $j.$

    Set $p_j = \chi_{\alpha_j E}(u_n)$. This is a projection in $M$. We set $w = \sum_{j=1}^N \alpha_j p_j$. Then $\|u_n - w\|_2 \leq \|u_n - w\|_\infty < \frac{3 \pi}{N}$. (We use the estimate $|e^{2\pi i \frac{1}{N}} - 1| < \frac{2\pi}{N}$, which holds for all $N\in \N$.)

    Since $M$ is diffuse, there are projections $q_j \in M$ such that the $q_j$ are mutually orthogonal, commute with the $p_j$, and for all $j,$ $\tau(q_j) = \frac{1}{N}$. We may also arrange for $|\tau(q_j-p_j)| < \frac{6}{N^4}$ for all $j$ and $q_j - p_j$ is either a projection or a projection times -1.

    Now set $w' = \sum_j \alpha_j q_j.$ We see that $\|w - w'\|_2 \leq \sum_{j=1}^N \|p_j-q_j\|_2 \leq N \sqrt{\frac{6}{N^4}} = \frac{\sqrt{6}}{N}$.

    Lastly, we need to show that the unitary $w'$ is close to a Haar unitary. Identify each $q_j$ with the projection onto the set $\alpha_j E.$ (This is an arc of length $1/N$ in $\mathbb T$ starting at $\alpha_j$). Under this identification, the unitary $v$ defined by $v(x) = x$ is a Haar unitary, and $\|v - w'\|_2 < \|v-w'\|_\infty < \frac{2\pi}{N}.$

    In total, we have that for all $\ee>0$, there is $A\in \cU$ such that for all $n\in A,$ there is a Haar unitary $v$ such that  $\|u_n - v\|_2 < \frac{2\pi}{N} + \frac{\sqrt{6}}{N} + \frac{2\pi}{N} < \ee. $
    
\end{proof}
We thank Jesse Peterson for suggesting the following second short proof in the factor case: 

\begin{proof}[Alternative proof of Lemma \ref{haar lifting OP} in the factor case]
    Suppose $M$ is a II$_1$ factor. Let $u\in \cH(M^\cU)$. Pick any $v\in \cH(M).$ By Proposition \ref{jung-atkinson} there is $w \in \cU(M^\cU)$ such that $v = wuw^*$. Lift $w$ to a sequence of unitaries $w_n \in \cU(M).$ Then, setting $u_n = w_n^*wuw^*w_n,$ we see that $u_n\in\cH(M)$ and $u = (u_n)_n \in M^\cU.$

\end{proof}

\begin{lem}
\label{lem-diff-lift}
    Let $(M,\tau)$ be a diffuse tracial von Neumann algebra. Say $x\in M$ is diffuse if $W^*(x)$ is diffuse. Then
    \begin{enumerate}
        \item The diffuse self-adjoints of $M$ are SOT-dense in $M_{sa}$.
        \item The diffuse unitaries in $M$ are SOT-dense in $\cU(M).$
    \end{enumerate}
\end{lem}
    
\begin{proof}
    For (1), it suffices to consider the case $M = L^\infty([0,1])$ as $W^*(x)$ will be contained in a diffuse abelian subalgebra. So let $f = f^* \in L^\infty([0,1])$, and fix $\ee > 0.$ 

    Claim: there exists $0\leq \delta < \ee$ such that $f(x) + \delta x$ is a diffuse element of $L^\infty([0,1])$.

    Proof of claim: Suppose not. Then for all $0\leq \delta <\ee$ there would be an atom in the spectrum of $f(x) + \delta x$. Say this atom is at $t_\delta$. Then for each $\delta,$ we have that the set $A_\delta = \{x: f(x)+\delta x = t_\delta\}$ has positive measure. Since there are uncountably many sets $A_\delta$, there must exist $\delta_1\neq \delta_2$ such that $B = A_{\delta_1}\cap A_{\delta_2}$ has positive measure. Then for $x\in B$, we have that $f(x) + \delta_1x = t_{\delta_1}$ and $f(x) + \delta_2x = t_{\delta_2}$. This implies that $(\delta_1-\delta_2)x$ is a constant function on $B,$ which is not possible. This ends the proof of the claim.

    To finish the proof of (1), we see that $\|f - (f+\delta x)\|_2 = \delta \|x\|_2 = \frac{\delta}{\sqrt{3}} < \ee$.

    To prove (2), observe that if $u$ is a unitary then $u = e^{ia}$ for some self-adjoint $a = a^*$ and apply (1).
\end{proof}

\subsection{1-bounded entropy}\label{1bdd entropy facts}

We recall some notation and properties of the Jung-Hayes 1-bounded entropy (\cite{JungSB}, \cite{Hayes2018}). For a diffuse tracial von Neumann algebra $(M,\tau)$ and $X\in M_{\text{sa}}^d$, the {\it law} of $X$ is the linear functional $\ell_X:\mathbb{C}\langle t_1,\dots,t_d\rangle \to \mathbb{C}$ given by $\ell_X (f) = \tau(f(X))$.
Let $\Sigma_{d,R}$ be the set of all linear maps $\ell: \mathbb{C}\langle t_1,\dots,t_d\rangle \to \mathbb{C}$ satisfying that there exists a tracial von Neumann algebra $(M,\tau)$ and $X \in M_{\text{sa}}^d$ such that $\ell=\ell_X$ and $\|x\| \leq R$ for all $x\in X$. We equip $\Sigma_{d,R}$ with the weak$^*$ topology.

Let  $X, Y\subset M_{\text{sa}}$ finite such that $\|x\| \leq R$ for all $x\in X\cup Y$.  Following \cite{VoiculescuFreeEntropy2}, for each weak$^*$ neighborhood $\mathcal{O}$ of $\ell_{X\sqcup Y}$ in $\Sigma_{d,R}$ and $n \in \mathbb{N}$, one defines
\[
\Gamma_R^{(n)}(X:Y; \mathcal{O}) =  \{A \in \mathbb{M}_n(\mathbb{C})_{\text{sa}}^{X}: \exists B\in \mathbb{M}_n(\mathbb{C})_{\text{sa}}^{Y} \text{ such that }  \ell_{A\sqcup B} \in \mathcal{O}, \|A_x\|,\|B_y\|\leq R, \forall x\in X,y\in Y \}.
\]

The $1$-bounded entropy of $N$ in the presence of $M$, denoted $h(N: M)$, where $N\subset M$ is a diffuse von Neumann subalgebra, is defined as a limiting quantity in terms of the exponential growth rate of covering numbers of $\Gamma_R^{(n)}(\mathcal{X,Y, \mathcal{O}})$ up to unitary conjugation as $n \to \infty$ for neighborhoods $\mathcal{O}$ of $\ell_x$ and $X\subset N$, 
 $Y\subset M$ finite. We write $h(N)$ in place of $h(N:N)$ when appropriate. We do not need the precise formula as we only need to consider this axiomatically for the purposes of our paper. The main properties we will need are below:   

  \begin{fact}(see  \cite[2.3.3]{HJKE1})\label{fact 1}
 $h(N_{1}:M_{1})\leq h(N_{2}:M_{2})$ if $N_{1}\subset N_{2}\subset M_{2}\subset M_{1}$ and $N_{1}$ is diffuse.
 \end{fact}

  \begin{fact}(see \cite[Proposition 4.5]{Hayes2018})\label{in the presence of the ultra}
 $h(N:M)=h(N:M^{\mathcal U})$ if $N\subset M$ is diffuse, and $\mathcal U$ is an ultrafilter on a set $I$. (Note that \cite[Proposition 4.5]{Hayes2018} asserts this fact for free ultrafilters $\mathcal U$. The fact is trivially true also for non-free (i.e., principal) ultrafilters.) \end{fact}

 \begin{fact}\label{unions}(see \cite[Lemma A.10]{Hayes2018})
  Assume that $(N_\alpha)_\alpha$ is an increasing chain of diffuse von Neumann subalgebras of $M$. Then $h(\bigvee_\alpha N_\alpha:M)=\sup_\alpha h(N_\alpha:M)$.
 \end{fact}

 \begin{fact}\label{joins}(see \cite[Lemma A.12]{Hayes2018})
  $h(N_1\vee N_2:M)\leq h(N_1:M)+h(N_2:M)$ if $N_1,N_2\subset M$ and $N_1\cap N_2$ is diffuse.  
  In particular, $h(N_1\vee N_2)\leq h(N_1)+h(N_2)$.
 \end{fact}
 
\begin{fact}\label{path fact 1bdd}
Assume that  $u_1,u_2\in \mathcal U(M)$ are such that there are Haar unitaries $v_1, \cdots, v_n\in \cH(M^{\mathcal{U}})$ satisfying $[u_1,v_1]= [v_i,v_{i+1}]=[v_n,u_2]=0$ for all $1\leq i\leq n-1$. Then $h(\{u_1,u_2\}'':M^{\mathcal{U}})\leq 0$.
\end{fact}

\begin{proof} Since $\{u_1,v_1\}'',\{v_1,v_2\}'',\cdots, \{v_{n-1},v_n\}'',\{v_n,u_2\}''$ are abelian, we get for each $1\leq i\leq n$, $$h(\{u_1,v_1\}'')=h(\{v_i,v_{i+1}\}'')=h(\{v_n,u_2\}'')=0.$$ Since $\{v_i\}''$  are diffuse for each $1\leq i\leq m$, Fact \ref{joins} implies that $$h(\{u_1,u_2,v_1,\cdots, v_n\}'')=h(\{u_1,v_1\}''\bigvee \{v_1,v_2\}''\bigvee \cdots\bigvee \{v_n,u_2\}'')\leq 0.$$ Hence, using Facts \ref{fact 1} and \ref{in the presence of the ultra} we see that $$h(\{u_1,u_2\}'':M)\leq h(\{u_1,u_2\}'': \{u_1,u_2,v_1,\cdots, v_n\}'')\leq h(\{u_1,u_2,v_1,\cdots, v_n\}'')\leq 0,$$
which proves the fact.
\end{proof}

Below we provide a more or less up to date list of examples of $h(N)>0$.

\begin{thm}\label{examples of h>0}
The following tracial von Neumann algebras $(N,\tau)$ satisfy $h(N)>0$. The first four examples all arise from identifying generating sets $X$ satisfying $\delta_0(X)>1$, and thus $h(N)=\infty$.  

\begin{enumerate}

    \item (see \cite[Lemma 3.7]{JungSB})) $N_1*N_2$ where $(N_1,\tau_1)$ and $(N_2,\tau_2)$ are Connes-embeddable diffuse tracial von Neumann algebras.
    \item The free perturbation algebras of Voiculescu (see Theorem 4.1 in \cite{brownimrn}). 
    \item Many examples of amalgamated free products  $N_1*_{B}N_2$ where $B$ is  amenable (see Section 4 of \cite{DykemaJungBrown} for precise examples).

    \item (see \cite{Shlyakhtenkononinnercocycles}, Theorem 3) Von Neumann algebras of Connes-embeddable nonamenable groups $\Gamma$ admitting non-inner cocycles $c:\Gamma\rightarrow\mathbb C\Gamma$. 
      \item (see \cite{HayesPT}, \cite{PTkilled}, \cite{bordenave2023norm}, \cite{HJKEPT}) Arbitrary nonamenable von Neumann subalgebras of $\text{L}(\mathbb{F}_t)$ for $t>0$.\label{pt 1bdd entropy theorem}
      \item {(see \cite{elayavalli2023remarks}) Von Neumann algebras arising from certain limit groups.}  
      \item {(see \cite[Theorem 1.1]{Jekeltypes})} Matrix ultraproducts {$\prod_{\mathcal U} \mathbb{M}_{n_k}(\mathbb{C})$}.
\end{enumerate}

\end{thm}

\subsection{Minor modification of the  \cite{exoticCIKE} construction} \label{CIKE construction}

We describe below a slightly modified version of the II$_1$ factor construction introduced in Section 4 of \cite{exoticCIKE}. 

\begin{defn}
    Let $N $ be a II$_1$ factor, and let $u_1,u_2\in \mathcal{H}(N)$ such that $\{u_1\}''\perp \{u_2\}''$. Let $n\geq2$ and let $v_1,\ldots,v_n$ be Haar unitaries each generating a copy of $L(\Z).$ Denote $N = \Phi_n^0(N,u_1,u_2)$ and denote by $\Phi_n^1(N,u_1,u_2)$ the von Neumann algebra $N *_{\{u_1\}''} \{u_1\}''\otimes\{v_1\}''$. For $2\leq i \leq n-1,$ define $\Phi_n^i(N,u_1,u_2) = \Phi_n^{i-1}(N,u_1,u_2) *_{\{v_{i-1}\}''} \{v_{i-1}\}''\otimes \{v_i\}''$. Finally, we define $\Phi_n(N,u_1,u_2) = \Phi_n^n(N,u_1,u_2) = \Phi_n^{n-1}(N,u_1,u_2) *_{\{v_{n-1},u_2\}''} \{v_{n-1},u_2\}''\otimes\{v_n\}''.$

\end{defn}

For a II$_1$ factor $M$, we denote by $\mathcal V(M)$ the set of pairs $(u_1,u_2)\in\mathcal U(M)\times\mathcal U(M)$ such that $u_1,u_2\in \mathcal{H(M)}$ and $\{u_1\}''\perp\{u_2\}''$. We endow $\mathcal U(M)\times\mathcal U(M)$ with the product $\|\cdot\|_2$-topology.

Let $M_1$ be a a II$_1$ factor. We construct a II$_1$ factor following \cite{exoticCIKE} which we denote by $\mathcal{S}_n(M_1)$ which contains $M_1$ and arises as the inductive limit of 
 a sequence $(M_k)_{k\in\mathbb N}$ of II$_1$ factors satisfying $M_k\subset M_{k+1}$, for every $k\in\mathbb N$. Let $\sigma=(\sigma_1,\sigma_2):\mathbb N\rightarrow\mathbb N\times\mathbb N$ be a bijection such that $\sigma_1(k)\leq k$, for every $k\in\mathbb N$. 
 Assume that $M_1,\ldots,M_k$ have been constructed, for some $k\in\mathbb N$.
Let $\{(u_1^{k,\ell},u_2^{k,\ell})\}_{\ell\in\mathbb N}\subset\mathcal V(M_k)$ be a $\|\cdot\|_2$-dense sequence. 
We define $$M_{k+1}:=\Phi_n(M_k,u_1^{\sigma(k)},u_2^{\sigma(k)}).$$
Note that $M_{k+1}$ is well-defined since $\sigma_1(k)\leq k$ and thus $(u_1^{\sigma(k)},u_2^{\sigma(k)})\in\mathcal V(M_k)$.
Then $M_k\subset M_{k+1}$, and we define $$\mathcal{S}_n(M_1):=({\cup_{k\in\mathbb N}M_k})''.$$ 

We remark that for a dense set of pairs of orthogonal Haar unitaries in $\cS_n(M_1),$ there are paths of length $n$ of commuting Haar unitaries in $\cS_n(M_1)$ connecting them. (Namely, the paths are formed by the unitaries $v_1,\ldots,v_n$ used in the above construction.) In the notation of Definition \ref{def-path}, $p(u_1,u_2) \leq n+1$ for a dense set of pairs of orthogonal Haar unitaries $(u_1,u_2).$ Lemma \ref{density lemma} will imply that $p(u_1,u_2)\leq n+1$ for all pairs of orthogonal Haar unitaries in $\cS_n(M_1)$ and in we will see in Theorem \ref{cike-one-orb} that in fact $p(u_1,u_2) \leq 2(n+1)$ for all pairs of Haar unitaries (not necessarily orthogonal) in $\cS_n(M_1).$

We now record some important facts about $\mathcal{S}_n(M_1)$. The proof of the proposition is adapted from the proof of Theorem 4.2 in \cite{exoticCIKE}.

\begin{prop}\label{factorlemmanew}
    If $M_1$ is a II$_1$ factor, then $M_1'\cap \cS_n(M_1) = \C$. In particular, $\cS_n(M_1)$ is a II$_1$ factor. 
\end{prop}

\begin{proof}
    It suffices to show that $M_1'\cap M_k = \C1$ for each $k\geq 1$, and in turn it suffices to show that $M_1'\cap \Phi_n^i(M_k,u_1^{\sigma(k)},u_2^{\sigma(k)}) = \C1$ for all $k\geq1$ and all $0\leq i \leq n$. 

    For the first case, consider any $k\geq1$ and $1\leq i\leq n-1.$ We write $\Phi_n^j$ in place of $\Phi_n^j(M_k,u_1^{\sigma(k)},u_2^{\sigma(k)})$ for ease of notation. We have $\Phi_n^i = \Phi_n^{i-1} *_{\{v_{i-1}\}''}\{v_{i-1}\}''\otimes \{v_i\}''$ for Haar unitaries $v_{i-1},v_i.$ Since $M_1$ is a II$_1$ factor, $M_1 \not\prec_{\Phi_n^{i-1}}\{v_{i-1}\}''$. By Theorem 1.1 of \cite{IPP08}, we get that $M_1'\cap \Phi_n^i \subset \Phi_n^{i-1}$. This implies by the induction hypothesis that $M_1'\cap \Phi_n^i = \C1.$

    Now consider when $i = 0.$ If $k = 1,$ this is the base case. Otherwise, $\Phi_n^0(M_k,u_1^{\sigma(k)},u_2^{\sigma(k)}) = M_k = \Phi_n^n(M_{k-1},u_1^{\sigma(k-1)},u_2^{\sigma(k-1)})$. So we may assume that $i = n$ and $k\geq1$ We again write $\Phi_n^j$ as shorthand for $\Phi_n^j(M_k,u_1^{\sigma(k)},u_2^{\sigma(k)})$. We write $u_1$ and $u_2$ as shorthand for $u_1^{\sigma(k)}$ and $u_2^{\sigma(k)}$ respectively. 

    If $n=2,$ our convention is that $v_0 = u_1$. We recall that $\Phi_n^{n-1} = \Phi_n^{n-2} *_{\{v_{n-2}\}''} \{v_{n-2}\}''\otimes\{v_{n-1}\}'' $, that $\{u_2\}''\perp\{v_{n-2}\}''$, and that $\{v_{n-1}\}''\perp \{v_{n-2}\}''.$ Since $M_1$ is a II$_1$ factor, $M_1\not\prec_{\Phi_n^{n-1}} \{v_{n-2}\}''$ and $M_1\not\prec_{\Phi_n^{n-2}} \{v_{n-1}\}''$. Lemma 4.6 of \cite{exoticCIKE} then says that $M_1\not\prec_{\Phi_{n}^{n-1}} \{v_{n-1}\}''\bigvee\{u_2\}'' = \{v_{n-1},u_2\}''.$

    We now recall that $\Phi_n^n = \Phi_n^{n-1} *_{\{v_{n-1},u_2\}''} \{v_{n-1},u_2\}''\otimes \{v_n\}''$. Applying Theorem 1.1 of \cite{IPP08}, we get that $M_1'\cap \Phi_n^n \subset \Phi_n^{n-1}$. By the induction hypothesis we get that $M_1'\cap \Phi_n^n = \C1$.
\end{proof}

\begin{cor}
    If $M_1$ is a II$_1$ with Property (T), then $\cS_n(M_1)$ does not have Property Gamma.
\end{cor}

\begin{proof}
    Since $M_1' \cap \cS_n(M_1) = \C1$ and $M_1$ has Property (T), we get that $\cS_n(M_1)'\cap \cS_n(M_1)^\cU \subset M_1'\cap \cS_n(M_1)^\cU = (M_1'\cap \cS_n(M_1))^\cU = \C1.$
\end{proof}

\subsection{Freeness and lifting in ultraproducts}

\begin{defn}\label{defn-indep}
    Let $(M,\tau)$ be a tracial von Neumann algebra. Two subsets $X,Y\subset M\ominus \C$ are said to be \emph{independent} if whenever $x\in X$ and $y\in Y$, we have $\tau(xy) = 0.$ This is also referred to as orthogonality in the literature, and sometimes written $X \perp Y.$

    Two subsets $X_1,X_2\in M\ominus \C$ are said to be \emph{freely independent} if whenever $x_k\in X_{i_k}$ for $1\leq k\leq n$, and $i_k\neq i_{k+1}$ for all $1\leq k\leq n-1,$ we have $\tau(x_1x_2\cdots x_n) = 0.$
\end{defn}

We will need the following results on asymptotic freeness in the ultrapowers. 

\begin{thm}[\cite{Popaindep}]\label{Popa freeness}
    Let $M_n$, $n\in \mathbb{N}$ be a sequence of  II$_1$ factors, $\mathcal{U}$ a non-principal ultrafilter on $\mathbb{N}$, and $A\subset \prod_{n\to \mathcal{U}} M_n$ a separable von Neumann subalgebra. Then there is a Haar unitary $u\in \prod_{n\to \mathcal{U}} M_n$ such that $u$ is freely independent from $A$. 
\end{thm}

Secondly, we need the recent free independence theorem of Houdayer-Ioana. The version we state is just a special case we need, following from Lemma 3.1 and Theorem A in \cite{houdayer2023asymptotic}.  

\begin{thm}[\cite{houdayer2023asymptotic}]\label{HI independence}
      Consider an amalgamated free product II$_1$  factor $M= M_1 *_{B} M_2$, and $u_i \in \mathcal{H}(M_i)$ such that $E_{B}(u_i^k)=0$ for every $k\in \mathbb{N}$ and $i\in \{1,2\}$. Suppose $v_i\in \mathcal{H}(M^\mathcal{U})$ such that $[v_i, u_i]=0$ and $E_{B^{\mathcal{U}}}(v_i)=0$ for each $i\in \{1,2\}$. Then $v_1,v_2$ are freely independent. 
\end{thm}

We need also a special case of the lifting theorem in Theorem 5.1 in \cite{houdayer2023asymptotic}.  

\begin{thm}[\cite{houdayer2023asymptotic}]\label{HI lifting}
    Let $w_1,w_2\in \mathcal{H}(N^\mathcal{U})$ be such that $w_1$ is freely independent from $w_2$. Then  there exist lifts $w_i= (w_i^{(n)})_{\mathcal{U}}$ such that $w_1^{(n)}$ is independent to $w_2^{(n)}$ for every $n\in \mathbb{N}$. 
\end{thm}

\subsection{Silver's Dichotomy}

We state a special case of Silver's Dichotomy \cite{silver1980counting} (see also Theorem 5.3.5 in \cite{gao2008invariant}).

\begin{thm}[Silver's Theorem]\label{silver-dichotomy}
    Let X be a standard Borel space and E a Borel equivalence relation on
    X. Then either there are countably many E-equivalence classes or there are
    continuum many E-equivalence classes.
\end{thm}

Silver's Theorem shows that whenever $(M,\tau)$ is separable and $\sim_M$ has uncountably many orbits, it in fact has exactly $\mathfrak c$ orbits.

\section{The equivalence relation $\sim_M$}

\subsection{Definitions and ultrafilters} Unless otherwise specified, $M$ denotes a diffuse tracial von Neumann algebra with trace $\tau.$

\begin{defn}
    Let $(M,\tau)$ be a diffuse tracial von Neumann algebra. Fix a countably cofinal ultrafilter $\cU$ (see Definition \ref{cofinal}).  Define $\sim_{M}$ to be the equivalence relation defined on $\cH(M)$ in the following way: we say $u\sim_M v$ if there are $w_1,\ldots,w_n \in \cH(M^\cU)$ such that $[u,w_1] =[w_k,w_{k+1}] = [w_n,v] = 0$ for all $1\leq k < n$. Denote by $\cO_M(u)$ all elements $v\in \cH(M)$ such that $v\sim_M u.$ 
\end{defn}

Now we review basic properties of ultrafilters, with the aim of showing that $\sim_M$ is independent of the choice of (countably cofinal) ultrafilter. This will allow one to determine non-elementary equivalence of factors.

A filter $\mathcal U$ on a set $S$ is a collection of subsets of $S$ that satisfies:
\begin{enumerate}
    \item $\mathcal U$ does not contain the empty set,
    \item if $X, Y \in \mathcal U$ then $X \cap Y \in \mathcal U$, and
    \item if $X\in \mathcal U$ and $X \subset Y \subset S$, then $Y \in \mathcal U$.
\end{enumerate}

An ultrafilter is a maximal filter, which exist by Zorn's Lemma. Ultrafilters have the property that for every subset $X\subset S$, either $X$ or $S\setminus X$ (but not both) is in $\mathcal U$. Every ultrafilter on an infinite set $S$ is either principal (i.e., for some $s\in S$, $\mathcal U = \{X \subset S : s\in X\}$) or free (i.e., $\mathcal U$ contains all cofinite sets $X\subset S$). $\mathcal U$ cannot be free and principal.

For the convenience of the reader, we include a proof of the most basic lifting lemma which is well-known to experts (see for instance \cite{Connes}).

\begin{lem}
    Let $(M,\tau)$ be a tracial von Neumann algebra and $\mathcal U$ an ultrafilter on a set $S$.
    \begin{enumerate}
        \item If $x \in M^{\mathcal U}$ is self-adjoint, then there are self-adjoint elements $x_s \in M$ such that $x = (x_s)_{s\in S}$. 
        \item If $u \in M^{\mathcal U}$ is a unitary, then there are unitaries $u_s \in M$ such that $u = (u_s)_{s\in S}$. 
    \end{enumerate}
\end{lem}

\begin{proof}
    \textcolor{white}{.}

    \begin{enumerate}
        \item Since $x$ is self-adjoint, $\lim_{s\to \mathcal U}\|x_s - x_s^*\|_2 = 0.$ Since $x_s - x_s^* = 2i \Img(x_s)$, we see that $\lim_{s\to\mathcal U}\|\Img(x_s)\|_2 = 0$ too. Therefore $(\Real(x_s))_{s\in S}$ and $(x_s)_{s\in S}$ both represent $x$. The former clearly only consists of self-adjoint elements.
        \item For $u\in M^{\mathcal U}$, write $u = e^{ix}$ for some self-adjoint $x\in M^{\mathcal U}$. Apply part (1) to $x$ to get self-adjoints in $M$ such that $x = (x_s)_s$. Since the continuous functional calculus commutes with all *-homomorphisms, we have that $\exp(i(x_s)_x) = (\exp(ix_s))_s = u.$ \qedhere
    \end{enumerate}
\end{proof}

\begin{defn}\label{cofinal}
    An ultrafilter $\mathcal U$ on a set $S$ is called \emph{countably cofinal} if there exists a sequence $(A_n)_{n\in\N}$ of sets in $\mathcal U$ such that $\cap_{n\geq 1} A_n = \emptyset.$ Otherwise, $\mathcal U$ is called \emph{countably complete}. 
\end{defn}

We remark that if $\mathcal U$ is countably complete, it follows that if $(A_n)_{n\in\N}$ is a sequence of sets in $\cU$ that $\cap_{n\geq 1} A_n \in \cU$. We note that every principal ultrafilter is countably complete and every free ultrafilter on $\N$ is countably cofinal. The existence of free countably complete ultrafilters (necessarily on a set larger than $\N$) is equivalent to the existence of measurable cardinals, a hypothesis independent of ZFC. See Section 5 of \cite{keisler2010ultraproduct}.

The following lemma is Lemma 2.3(2) of \cite{BCI15}.

\begin{lem}
    If $(M,\tau)$ is a separable tracial von Neumann algebra and $\mathcal U$ is countably complete then $M$ is *-isomorphic to $M^{\mathcal U}$.
\end{lem}

Following \cite{BCI15}, we now consider a countably cofinal ultrafilter $\mathcal U$ on an infinite set $S$. Fix a sequence $(A_n)_{n\geq2}$ of sets in $\mathcal U$ so that $\cap_{n\geq2}A_n = \emptyset$. Set $A_1 = S.$ Define $B_n = \cap_{1\leq k\leq n} A_k$ so that each $B_n \in \mathcal U$, $\cap_{n\geq1}B_n = \emptyset$, and the $B_n$ form a decreasing sequence.

Define a function $f:S\to\mathbb N$ so that $f(s)$ is the largest integer $n$ so that $s \in B_n.$ This is well defined, and $\lim_{s\to\mathcal U} f(s) = \infty.$ Indeed, $B_n = f^{-1}([n,\infty))$.

The following proposition now implies that our equivalence relation is independent of the choice of countably cofinal ultrafilter:

\begin{prop}\label{sim-well-def}
    Let $S$ be an infinite set and $\mathcal U$ a countably cofinal ultrafilter on $S$. Let $(M,\tau)$ be a diffuse tracial von Neumann algebra and $u,v \in \cH(M).$ Then for each $k\geq0$, TFAE:
    \begin{enumerate}
        \item There exist $w_1,\ldots,w_k \in \cH(M^\mathcal U)$ such that, setting $w_0 = u$ and $w_{k+1} = v$, $[w_i,w_{i+1}] = 0$ for all $0\leq i\leq k.$
        \item There exist diffuse unitaries (resp. self-adjoints) $w_1,\ldots,w_k \in M^\cU$ such that, setting $w_0 = u$ and $w_{k+1} = v$, $[w_i,w_{i+1}] = 0$ for all $0\leq i\leq k.$
        \item For all $n\geq 1,$ there exist $w_1,\ldots,w_k \in \mathcal H(M)$ such that $\|[w_i,w_{i+1}]\|_2 < 1/n$ for all $0\leq i\leq k$, where $w_0 = u$ and $w_{k+1} = v$.
        
    \end{enumerate}
    
\end{prop}

\begin{proof}
    (1) implies (2) is clear since Haar unitaries are diffuse. (To get a diffuse self-adjoint, write $w_i = e^{i a_i}$ for a self-adjoint $a_i$.)

    (2) implies (1) is because if $w_i$ is diffuse then $W^*(w_i) \simeq L\mathbb Z$ so we can replace $w_i$ with the generator of $L\mathbb Z$.

    (1) implies (3): Take $w_1,\ldots,w_k \in \cH(M^\mathcal U)$ as in (1). Using Lemma \ref{haar lifting OP}, write $w_i = (w_{i,s})_{s\in S}$ where $w_{i,s} \in \mathcal H(M)$ for all $i$ and all $s.$ By definition of the ultrapower, we have that $\lim_{s\in\mathcal U}\|[w_{i,s},w_{i+1,s}]\|_2 = 0$ for all $0\leq i\leq k$ where $w_0 = u$ and $w_{k+1} = v.$ Then given $n\in \mathbb N$ we can find $s_n \in S$ such that for all $0\leq i\leq k$ $\|[w_{i,s_n},w_{i+1,s_n}]\|_2 < 1/n$. In other words, we have (3).

    (3) implies (1): Suppose for each $n\geq 1$ and $1\leq i\leq k$ we have $w_{i,n}\in\cH(M)$ as in (3). Set $w_{0,n} = u$ and $w_{k+1,n} = v$ for all $n.$ Recall the function $f:S\to \mathbb N$ defined before the proposition. For $0\leq i\leq k+1$ and $s\in S$, define $x_{i,s} := w_{i,f(s)} \in \mathcal H(M)$, and define $x_i = (x_{i,s})_{s\in S} \in \mathcal H(M^\mathcal U)$. Then for all $0\leq i\leq k$, we have $\|[x_i,x_{i+1}]\|_2 = \lim_{s\to\mathcal U}\|[x_{i,s},x_{(i+1),s}]\|_2 = \lim_{s\to\mathcal U}\|[w_{i,f(s)},w_{(i+1),f(s)}]\|_2.$ Since $\lim_{s\to\mathcal U}f(s) = \infty,$ this implies that $\lim_{s\to\mathcal U}\|[w_{i,f(s)},w_{(i+1),f(s)}]\|_2 = \lim_{n\to\infty}\|[w_{i,n},w_{(i+1),n}]\|_2 = 0.$ Therefore, $[x_i,x_{i+1}] = 0$ for all $0\leq i\leq k$. 
\end{proof}

We can also prove a saturation result, saying that the relation $\sim_M$ is stable under taking iterated ultrapowers. We first define a variant of $\sim_M$ with no reference to an ultrapower.

\begin{defn}\label{eq-rel-no-uf}
        Let $(M,\tau)$ be a diffuse tracial von Neumann algebra. Denote by $\sim_{M}^0$ the equivalence relation defined on $H(M)$ in the following way: we say $u\sim_{M}^0 v$ if there are $w_1,\ldots,w_n \in \cH(M)$ such that $[u,w_1] =[w_k,w_{k+1}] = [w_n,v] = 0$ for all $1\leq k < n$. 
\end{defn}

\begin{prop}
\label{prop-u.p-of-u.p}
    If $\mathcal U$ is a countably cofinal ultrafilter on $S$ then $\sim_{M^\mathcal U}^0 = \sim_{M^\mathcal U}$.
\end{prop}

\begin{defn}
    If $\mathcal U$ is an ultrafilter on $S$ and $\mathcal V$ is an ultrafilter on $T$ then $\mathcal U \rtimes \mathcal V$ is defined by $C \in\cU\rtimes\cV $ if and only if $\{s : \{t: (s,t) \in X\} \in \cV\}\in \cU$. 
\end{defn}

We observe that the above definition yields an ultrafilter on $S\times T$. We remark that there is not a symmetric notion of tensor product of ultrafilters that gives ultrafilters except in certain scenarios involving large cardinals. See Section 1 of \cite{goldberg2021products} for more details.

The following lemma appears as Theorem 1.3 in \cite{capraro2012product}.

\begin{lem}
    If $\mathcal U$ is an ultrafilter on $S$ and $\mathcal V$ is an ultrafilter on $T$ and $f:S\times T \to \mathbb C$ is a bounded function then $$\lim_{s\to\mathcal U} \lim_{t\to \mathcal V} f(s,t) = \lim_{(s,t)\to\mathcal U\rtimes \mathcal V} f(s,t).$$
\end{lem}

The following lemma appears as Theorem 2.1 in \cite{capraro2012product}.

\begin{lem}
    $(M^\mathcal U)^ \mathcal V \simeq M^{\mathcal U \rtimes \mathcal V}$.
\end{lem}

\begin{proof}[Proof of Proposition \ref{prop-u.p-of-u.p}]
    It is readily shown from the definitions and Lemma \ref{haar lifting OP} that $u\sim_{M^\mathcal U}^0 v$ is equivalent to:
    \begin{enumerate}
        \item There exists $k\geq0$ and $w_{i,s} \in \mathcal H(M)$ (where $w_{0,s} = u_s$ and $w_{k+1,s} = v_s$ come from the lifts of $u$ and $v$) such that for all $n\geq 1,$ there is $A_n \in \mathcal U$ such that for all $s\in A_n$, $\|[w_{i,s},w_{i+1,s}]\|_2 < 1/n$.
    \end{enumerate}
    Similarly, $u\sim_{M^\mathcal U} v$ is equivalent to:
    \begin{enumerate}
        \item[(2)] There is a countably cofinal ultrafilter $\cV$ on a set $T$ such that there exists $k\geq0$ and $x_{i,(s,t)} \in \mathcal H(M)$ (where $x_{0,(s,t)} = u_s$ and $x_{k+1,(s,t)} = v_s$ come from the lifts of $u$ and $v$) such that for all $n\geq 1,$ there is $C_n \in \mathcal U \rtimes \mathcal V$ such that for all $(s,t)\in C_n$, $\|[x_{i,(s,t)},x_{i+1,(s,t)}]\|_2 < 1/n$.
    \end{enumerate}

    Let us now prove that (1) and (2) are equivalent.

    If (1) holds, then pick any countably cofinal ultrafilter $\cV$ on a set $T$. Set $C_n = A_n\times T$ and $x_{i,(s,t)} = w_{i,s}$. Then (2) holds.

    Conversely, we will need to use the cofinality of $\mathcal U.$ Assume (2) holds. There are $A_n \in \mathcal U$ and $B_{s,n} \in \mathcal V$ such that $\cup_{s\in A_n}\{s\}\times B_{s,n}\subset C_n$ Since $\mathcal U$ is a filter, it is closed under intersections, and so we can choose the $A_n$ to be a decreasing sequences. Furthermore, since $\mathcal U$ is countably cofinal, we may take $A_n$ such that $\cap_n A_n = \emptyset$. Then for each $s\in A_1,$ there is a unique $n\in\mathbb N$ such that $s \in A_n\setminus A_{n-1}$. Take $t_s$ to be an element of $B_{s,n}.$ Now set $w_{i,s} = x_{i,(s,t_s)}$. With this choice, (1) holds.
    
\end{proof}

We emphasize that $\cU$ being countably cofinal in the above proof of (2) implies (1) is crucial. Indeed, if $s\in \cap_n A_n$ there is no good way to choose $t_s.$ This is because we would have to have some $t_s$ where $\|[x_{i,(s,t_s)},x_{i+1,(s,t_s)}]\|_2$ is exactly 0.

Henceforth, we will assume $\cU$ is a free (and therefore countably cofinal) ultrafilter on $\N$ unless otherwise specified. This assumption is justified by Proposition \ref{sim-well-def}.

\subsection{Basic properties of $\sim_M$}

In this article we will interchangeably use the words ``orbit'' and ``equivalence classes''. It is immediate to observe that $\cO_M(u)$ is closed under adjoints, inverses, conjugation by elements in $\cO_M(u)$, and nonzero integer powers. 
The equivalence relation $\sim_M$ is not countable but we will show in this subsection that it is Borel. In fact, $\sim_M$ is $F_\sigma.$

We will need the following notation. 
\begin{defn}\label{def-path}
    Let $u,v \in \cH(M).$ Then $p_M(u,v)$ denotes the length of the shortest possible ``commutator path'' between $u$ and $v$ in $M^\cU.$ We define $p_M(u,u) = 0.$ If $[u,v]=0$ and $u\neq v$ then $p_M(u,v) = 1.$ If there is $w\in \cH(M^\cU)$ such that $[u,w]=[w,v]=0$ then $p_M(u,v)\leq2.$ If $u\not\sim_M v$ then $p_M(u,v) = \infty.$
\end{defn}

We will need the following key ultrapower density lemma in some of our theorems: 
\begin{lem}\label{density lemma}
    If $p_M(u_n,v_n) \leq K $ for all $n\in\mathbb N,$ $u_n\to u \in \cH(M)$ in 2-norm, and $v_n\to v\in \cH(M)$ in 2-norm, then $p_M(u,v) \leq K$ too.
\end{lem}

\begin{proof}
    Without loss of generality, assume that $\|u_n-u\|_2 < 1/n$ and $\|v_n-v\|_2 < 1/n$ for all $n.$

    By the first assumption about path length, we know that for all $n$ there exist Haar unitaries $w_{1,n},\ldots,w_{K,n} \in \cH(M^\cU)$ such that for all $1\leq i\leq K-1$
    $$[u_n,w_{1,n}] = [w_{i,n},w_{i+1,n}] = [w_{K,n},v_n] = 0. $$

    By Lemma \ref{haar lifting OP}, we may write $w_{i,n} = (w_{i,n,m})_m$ where $w_{i,n,m} \in \mathcal H(M)$ are all Haar unitaries in $M.$ Define $w_{0,n} = w_{0,n,m} = u_n$ and $w_{K+1,n} = w_{K+1,n,m} = v_n$.

    By the ultrapower construction, we now know that for all $n\in\mathbb N,$ $0\leq i \leq K+1$, and $0\neq j\in\mathbb Z$,
    \begin{align*}
        \lim_{m\to\cU} \|[w_{i,n,m},w_{n,i+1,m}]\|_2 = 0
    \end{align*}
    Therefore for each $n\in\mathbb N$ there is $m(n) \in \mathbb N$ such that for all $0\leq i\leq K+1$,
    \begin{align*}
        \|[w_{i,n,m(n)},w_{i+1,n,m(n)}]\|_2 < 1/n.
    \end{align*}
    For $1\leq i\leq K,$ now set $w_i = (w_{i,n,m(n)})_{n\in\N}$. This is clearly a Haar unitary in $M^\cU$ since each $w_{i,n,m(n)}$ is a Haar unitary. Also, $\|[w_i,w_{i+1}]\|_2 = \lim_{n\to\cU}\|[w_{i,n,m(n)},w_{i+1,n,m(n)}]\|_2 = 0$ and hence $[w_i,w_{i+1}] = 0$ for all $1\leq i\leq K.$

    Now, fix $\varepsilon > 0.$ Pick $n$ such that $3/n < \varepsilon.$ Then 
    \begin{align*}
        \|[u,w_{1,n,m(n)}]\|_2 &\leq \|[u-u_n,w_{1,n,m(n)}]\|_2 + \|[u_n,w_{1,n,m(n)}]\|_2\\
        &\leq 2\|u-u_n\|_2\|w_{1,n,m(n)}\| + \|[w_{0,n,m(n)},w_{1,n,m(n)}]\|_2\\
        &< 2\|u-u_n\|_2\|w_{1,n,m(n)}\| + 1/n < 3/n < \varepsilon.
    \end{align*}
    Therefore $\|[u,w_1]\|_2 = \lim_{n\to\cU}\|[u,w_{1,n,m(n)}]\|_2 = 0.$ Similarly, $v$ commutes with $w_K.$

    Hence there is a path of length $K$ from $u$ to $v$ via $w_1,\ldots,w_K$, witnessing that $p_M(u,v) \leq K.$
\end{proof} 

\begin{prop}\label{Borel}
    The equivalence relation $\sim_M$ is a Borel equivalence relation on $\cH(M).$ Moreover, $\sim_M$ is $F_\sigma.$ Moreover, if $\fR(M) < \infty,$ then $\sim_M$ is closed.
\end{prop}

\begin{proof}
    Define $\sim_{M,k} = \{(u,v)\in\cH(M)\times\cH(M) : p_M(u,v)\leq k\}$. By Lemma \ref{density lemma}, we have that $\sim_{M,k}$ is closed. Clearly $\sim_M$ is the union of $\sim_{M,k}$ for all $k\geq1,$ showing that $\sim_M$ is $F_\sigma$ and thus Borel.
\end{proof}

\section{Single orbit and  property Gamma}

\begin{prop}
\label{prop-gamma-has-1-orbit}
    If $M$ is a II$_1$ factor with property Gamma then $\sim_M$ has one orbit.
\end{prop}

\begin{proof}
    If $M$ has Gamma then $M'\cap M^\cU$ is a diffuse tracial von Neumann algebra. Then any Haar unitary in $M'\cap M^\cU$ witnesses a path of length 2 between any two Haar unitaries $u,v\in \cH(M)$.
\end{proof}

\begin{thm}\label{char-gamma}
    A II$_1$ factor $M$ has property Gamma if and only if for all pairs of Haar unitaries $u_1,u_2\in \cH(M),$ $p_M(u_1,u_2) \leq 2.$
\end{thm}

\begin{proof}
    The proof of Proposition \ref{prop-gamma-has-1-orbit} shows the forward implication. Conversely, suppose that for all pairs of Haar unitaries $u_1,u_2\in \cH(M)$ there is $v\in \cH(M^\cU)$ such that $[u_1,v]=[v,u_2]=0.$ By Theorem \ref{shenshi} it suffices to show that for all $x\in M,$ $W^*(x)'\cap M^\cU$ is diffuse. Note that $x$ can be written as $x = a_1 + ia_2$ for two self-adjoint elements $a_1,a_2\in M$. Then $W^*(a_1)$ and $W^*(a_2)$ are abelian subalgebras of $M,$ and are therefore contained in diffuse abelian subalgebras. Let $u_1$ be a Haar unitary that generates a diffuse abelian subalgebra containing $W^*(a_1)$ and $u_2$ be a Haar unitary that generates a diffuse abelian subalgebra containing $W^*(a_2)$. By assumption, there is $v\in \cH(M^\cU)$ that commutes with $u_1$ and $u_2$. Therefore $v$ commutes with $a_1\in W^*(u_1)$ and $a_2\in W^*(u_2)$. Hence $v$ commutes with $x = a_1 + ia_2.$ So $W^*(v) \subset W^*(x)'\cap M^\cU$. This demonstrates that $W^*(x)'\cap M^\cU\neq\C1$.
\end{proof}

We will now demonstrate how to construct many non-Gamma factors such that $\sim_M$ admits only 1 orbit.

\begin{thm}\label{cike-one-orb}
    There exists a non-Gamma II$_1$ factor $M$ such that $\sim_M$ has one orbit and further $p_M(u,v)\leq 6$ for any $u, v\in\cH(M)$.
\end{thm}

\begin{proof}
Recall the \cite{exoticCIKE} construction from Section \ref{CIKE construction}, and fix $M=\mathcal{S}_2(L(SL_3(\mathbb{Z})))$.  Observe that if we fix unitaries $u_1,u_2 \in M^\cU$, then by Theorem \ref{Popa freeness}, taking  $A = \{u_1,u_2\}''$, there is a diffuse abelian von Neumann algebra $\{v\}''$ (where $v$ is a Haar unitary in $M^\cU$) such that $A$ and $\{v\}''$ are freely independent. In particular by  Theorem \ref{HI lifting}, we can lift $v=(v_i)_{\mathcal{U}}$ so that $u_1$ is independent to $v_i$ for each $i$. Hence by Lemma \ref{density lemma} and the comments preceding Proposition \ref{factorlemmanew} we have $u_1,v$ are related by a path of length 3. Applying the same idea for $u_2$     we see $u_1\sim_{M^\cU} u_2$ via a path of length at most 6. Note that the equivalence relation $\sim_{M^\cU}$ on $\cH(M^\cU)$ only has one orbit which in particular implies that $\sim_M$ also only has one orbit.
\end{proof}

We denote the maximum commutation path length between two Haar unitaries in $N$ by $\mathfrak{R}(N)$; see Definition \ref{defn-diameter} for a precise definition. The above proof shows the following computation: 

\begin{prop}\label{upperbound}
    We have $\mathfrak{R}(\mathcal{S}_n(N))\leq 2(n+1)$ for any II$_1$ factor $N$ and natural number $n\in \mathbb{N}$. 
\end{prop}

\section{Uncountably  many orbits and free group factors}

\begin{lem}
    Let $(M,\tau)$ be a diffuse tracial von Neumann algebra and $u\in \cH(M)$. Then $h(\{\cO_M(u)\}'': M) \leq 0.$
\end{lem}

\begin{proof}
    We use the facts on 1-bounded entropy outlined in Section \ref{1bdd entropy facts}. First note that $h(\cO_M(u)'':M) = h(\cO_M(u)'':M^{\mathcal{U}})$ by Fact \ref{in the presence of the ultra}. Moreover, by Fact \ref{path fact 1bdd}, for any Haar unitary $v\sim_{M} u$  we have $h(\{u,v\}'':M^{\mathcal{U}})\leq 0$. Therefore taking joins over all such $v$, combining Facts \ref{unions} and \ref{joins} we have that $h(\{\cO_{M}(u)\}'':M^\cU) \leq 0$, which concludes the proof.
\end{proof}

\begin{lem}
    Let $(M,\tau)$ be a diffuse tracial von Neumann algebra and let $S\subset \cH(M)$ consist of one representative of each orbit of $\sim_M$. Then $\bigcup_{u\in S} (\cO_M(u)'')_{sa} = M_{sa}$. Moreover, we have at least one of the two following situations:
    \begin{enumerate}
        \item There is $u\in \cH(M)$ such that $\cO_M(u)''=M$.
        \item $S$ is uncountable, i.e., $\sim_M$ has uncountably many orbits.
    \end{enumerate}
\end{lem}

\begin{proof}
    It suffices to show that for each $x\in M_{sa}$, there is $u\in \cH(M)$ such that $x\in \{u\}''.$ But this is clear since $\{x\}''$ is abelian and therefore contained in a diffuse abelian subalgebra since $M$ is diffuse, and diffuse abelian subalgebras are generated by a Haar unitary. For the moreover part, note that the Baire Category Theorem implies that if there are only countably many orbits, then one of the $(\cO_M(u)'')_{sa}$ be non-meager in $M_{sa}$. But every proper subspace is meager, hence $(\cO_M(u)'')_{sa} = M_{sa}$ and thus $\cO_M(u)''= M.$
\end{proof}

\begin{thm}\label{thm-uncntble-orbits}
    Let $M$ be type II$_1$ von Neumann algebra. If $h(M) > 0,$ then $\sim_M$ has uncountably many orbits. 
\end{thm}

\begin{proof}
Since $h(\cO_M(u)'') \leq 0$ for all $u\in \cH(M),$ we cannot have $\cO_M(u)'' = M$. By the previous lemma, we get that $\sim_M$ has uncountably many orbits.

\end{proof}

\begin{thm}\label{thm-FGF}
    Set $M = L(\mathbb F_2)$. Then there is a one-to-one correspondence between maximal amenable subalgebras of $M$ and orbits of $\sim_M$. Precisely, for all $u\in \cH(L(\mathbb F_2)),$ $\cO_M(u) = \cH(\cO_M(u)'')$ and $\cO_M(u)''$ is the unique maximal amenable subalgebra containing $u$.
\end{thm}

\begin{proof}
    If $A$ is a maximal amenable subalgebra, then $A$ only intersects one orbit of $\sim_M$ since $A$ is amenable (and thus has property Gamma). Conversely, if $u\in \cH(M)$ then $B=\cO_M(u)''$ is a subalgebra of $M$ with $h(B)=0$. Therefore $B$ is amenable by Theorem \ref{examples of h>0} (\ref{pt 1bdd entropy theorem}). $B$ is also clearly diffuse since $u$ is a Haar unitary. Hence $B$ is contained inside a unique maximal amenable subalgebra $A$. It is clear that $\cO_M(u) \subseteq \cH(\cO_M(u)'').$ Conversely, note that $\cO_M(u)''$ is amenable so it only contains one orbit, namely $\cO_M(u).$
\end{proof}

\section{A numerical invariant for tracial diffuse von Neumann algebras}

\subsection{Definitions and properties of commutation diameter and number of orbits}

\begin{defn}\label{defn-diameter}
    Let $(M,\tau)$ be a tracial diffuse von Neumann algebra. For an orbit $\mathcal O\subset \cH(M)$ of $\sim_M$, define $\fR(\mathcal O) := \sup_{u,v\in\cO} p_M(u,v)$ (This supremum could, at least in principle, be infinite.) Define by $\fR(M)$ the supremum of $\fR(\mathcal O) $ over all orbits $\mathcal O$ of $\sim_{M}$. We also define $\fO(M)$ to be the number of orbits of $\sim_M$ in $\cH(M).$ 
\end{defn}

The following is immediate from the definitions and independence of the choice of ultrafilter (see Proposition \ref{sim-well-def}).

\begin{prop}
    Let $(M,\tau)$ be a tracial diffuse von Neumann algebra and $\cU$ an ultrafilter. Then
    \begin{enumerate}
        \item $\fR(M) \leq \fR(M^\cU)$;
        \item $\fO(M) \leq \fO(M^\cU)$.
    \end{enumerate}
\end{prop}

Now we prove the following result stated as Theorem \ref{mainprop list}: (1). 
\begin{thm}\label{one-orbit}
    Let $(M,\tau)$ be a diffuse tracial von Neumann algebra. Then $\fO(M) = 1$ or $\fO(M) \geq \aleph_1.$ Moreover, if $M$ is separable, then $\fO(M) = 1$ or $\fO(M) = \mathfrak c$. Moreover, if $\fO(M) = 1$ then $\fR(M) < \infty.$
\end{thm}

\begin{proof}
     Let us assume that $\fO(M) \leq \aleph_0$.  By the Baire Category Theorem, there is $u\in \cH(M)$ such that $\cO_M(u)$ is non-meager. In fact, defining $\cO_{M,k}(u) = \{v\in \cH(M) : p_M(u,v) \leq k\},$ we have that $\cO_{M,k}$ is non-meager for some $k.$ Therefore there is $v\in \cO_M(u)$ and $\ee>0$ such that for all $w\in \cH(M)$, $\|v-w\|_2 < \ee$ implies $p_M(u,w)\leq k.$

     There are central projections $z_n\in Z(M),$ $n\geq0$ such that $Mz_0$ is either $\{0\}$ or has diffuse center and $Mz_n$ is either $\{0\}$ or a II$_1$ factor for each $n\geq 1.$ We assume for ease of presentation that all $z_n$ are non-zero; the case where some are zero is essentially the same.

     Let $w \in \cH(M)$ be arbitrary. Our goal now is to show that $w\in\cO_M(u)$. This will prove that $\fO(M) = 1.$ By Proposition \ref{jung-atkinson}, for each $n\geq1,$ there exists $x_n = (x_{n,r})_r \in \cU(M^\cU z_n)$ such that $x_n vz_n x_n^* \in W^*(wz_n)$. In particular, we have that there is $y_n \in W^*(wz_n)\subset Mz_n$ such that $x_nvz_nx_n^* = y_n.$

    We may find self-adjoint operators $a_{n,r} \in Mz_n$ such that $\|a_{n,r}\|_\infty \leq \pi$ and $x_{n,r} = e^{ia_{n,r}}$. Pick a natural number $m$ such that $\frac{2\pi}{m}< \ee$. Then $\|e^{ia_{n,r}/m} - 1\|_\infty < \ee/2$, and so $\|e^{ia_{n,r}/m}vz_ne^{-ia_{n,r}/m} - vz_n\|_\infty < \ee$ for all $n\geq1$ and $r\in\N.$

    Therefore 
    \begin{align*}
            \left\|vz_0 + \sum_{n\geq1} e^{ia_{n,r}/m}vz_ne^{-ia_{n,r}/m} - v\right\|_2^2 &= \sum_{n\geq1} \|e^{ia_{n,r}/m}vz_ne^{-ia_{n,r}/m} - vz_n\|_2^2  \\
            &< \sum_{n\geq1} \ee \|z_n\|_2^2 \\
            &= \sum_{n\geq1} \ee\tau(z_n) \leq \ee.
    \end{align*}
Furthermore, since we are simply conjugating $v$ by the unitary $z_0 + \sum_{n\geq1}e^{ia_{n,r}/m}z_n$, we see that $vz_0 + \sum_{n\geq1}e^{ia_{n,r}/m}vz_ne^{-ia_{n,r}/m}$ is a Haar unitary in $M$. Therefore, $vz_0 + \sum_{n\geq1}e^{ia_{n,r}/m}vz_ne^{-ia_{n,r}/m} \in \cO_{M,k}(u).$ Hence $p_M(vz_0 + \sum_{n\geq1}e^{ia_{n,r}/m}vz_ne^{-ia_{n,r}/m},v) \leq 2k.$ Conjugating repeatedly, we get that $$p_M(vz_0 + \sum_{n\geq1}e^{ia_{n,r}}vz_ne^{-ia_{n,r}},v)\leq 2km.$$ 

In other words, $p_M(vz_0+\sum_{n\geq1}x_{n,r}vz_nx_{n,r}^*,v )\leq 2km.$ This means that for all $r\in\N,$ there are Haar unitaries $b_{1,r},\ldots,b_{2km-1,r}\in\cH(M)$ such that $vz_0 + \sum_{n\geq1}x_{n,r}vz_nx_{n,r}^*$ commutes with $b_{1,r}$, which commutes with $b_{2,r}$, etc and $b_{2km-1,r}$ commutes with $v.$ Therefore, for all $n \geq 1$ and all $r\in\N,$ we have that $ x_{n,r}vz_nx_{n,r}^*$ commutes with $b_{1,r}z_n$, which commutes with $b_{2,r}z_n$, etc and $b_{2km-1,r}z_n$ commutes with $vz_n.$ The $b_{1,r}z_n$ are no longer necessarily Haar unitaries in $M^\cU z_n,$ but they are still diffuse unitaries. Therefore we can find $c_{i,n,r} \in \cH(M^\cU z_n)$ such that $W^*(c_{i,n,r}) = W^*(b_{i,r}z_n)$ for all $1\leq i\leq 2km-1,$ $n\geq 1$, and $r\in\N,$ and we have, as always, that $x_{n,r}vz_nx_{n,r}^*$ commutes with $c_{1,n,r}$, which commutes with $c_{2,n,r}$, etc and $c_{2km-1,n,r}z_n$ commutes with $vz_n.$ 

Furthermore, for each $n\geq1,$ $\|x_{n,r}vz_nx_{n,r}^* - y_n\|_2 \to 0$ as $r\to\cU.$ By the proof of Lemma \ref{density lemma}, there are $c_{i,n} \in \cH(M^\cU z_n)$ such that $y_n\in Mz_n$ commutes with $c_{1,n}$, which commutes with $c_{2,n}$, etc and $c_{2km-1,n}$ commutes with $vz_n.$

Let $c_0$ be a diffuse unitary in the center of $Mz_0.$ Then $c_0 + \sum_{n\geq1} c_{i,n}$ is a path of diffuse unitaries in $M^\cU$ connecting $c_0 + \sum_{n\geq1}y_n$ to $v.$ Lastly, since $y_n \in W^*(wz_n)$ for $n\geq 1$, $w$ commutes with $c_0 + \sum_{n\geq 1}y_n.$ Thus $w \sim_M c_0 + \sum_{n\geq1}y_n \sim_M v \sim_M u.$ 

The above argument shows that when there are countably many orbits, there is only one orbit, and there exist $u\in\cH(M)$ and $k,m\in\N$ such that $p_M(w,u) \leq 1 + 2km + k$ for any $w\in\cH(M)$. We note that $k$ and $m$ did not depend on the choice of $w,$ showing the second moreover. 

The first moreover follows from Theorem \ref{silver-dichotomy}.
    
\end{proof}

The proof of the following lemma is identical to the proof of Lemma \ref{density lemma}, following a diagonalization argument.  

\begin{lem}\label{u.p.-density-lem}
    Let $(M,\tau)$ be tracial, diffuse. Suppose $(u_n)_n$ and $(v_n)_n$ are sequences of unitaries in $M$ with $p_M (u_n,v_n)\leq k$ for all $n.$ Set $u = (u_n)_n \in M^\cU$ and $v=(v_n)_n \in M^\cU$. Then $p_{M^\cU}(u,v)\leq k$.
\end{lem}

\begin{thm}
    Assume that $\fO(M) =1.$ Then $ \fO(M^\cU) = 1$ and $\fR(M) = \fR(M^\cU)$.
    
\end{thm}

\begin{proof}
    By Theorem \ref{one-orbit}, we have that $\fR(M) < \infty.$

    Suppose $u\in \cH(M^\cU)$. Using Lemma \ref{haar lifting OP}, lift to Haar unitaries $(u_n)_n$. Since there is only one orbit of $\sim_M$, all of the $u_n$ are in one orbit, say $\mathcal O(v)$. Therefore we have commutator paths (of length at most $\fR(M)$) in $\cH(M^\cU)$ from $u_n$ to $v$ for all $n.$ Use Lemma \ref{u.p.-density-lem} to get a path of length at most $\fR(M)$ from $u$ to $v$.

    Since every $u\in \cH(M^\cU)$ is equivalent to a unitary from $\cH(M),$ we see that $\sim_{M^\cU}$ also only has one orbit.

    Now suppose $u\sim_{M^\cU} v$ in $\cH(M^\cU)$. We know that there is $w \in \cH(M)$ such that $u\sim_{M^\cU} w$ and $v\sim_{M^\cU} w$ from above. That means we can lift $u$ and $v$ to Haar unitaries $(u_n)_n$ and $(v_n)_n$ such that $u_n\sim_M w \sim_M v_n$ for all $n$. Therefore there are commutator paths of length at most $\fR(M)$ between $u_n$ and $v_n$ for each $n.$ Use Lemma \ref{u.p.-density-lem} to get a path of length at most $\fR(M)$ between $u$ and $v.$ Therefore $\fR(M^\cU)$ is bounded above by $\fR(M).$ 
\end{proof}
 
\subsection{Graph products having large commutation diameter}

We shall now prove the following result: 

\begin{thm}\label{thm-graph-radius}
        Let $M$ be a  graph product of diffuse tracial von Neumann  algebras where the underlying graph is   connected and has diameter at least 4. Then $\mathfrak{R}(M)\geq 4$ and moreover $M$ is not elementarily equivalent to the \cite{exoticCIKE} non-Gamma II$_1$ factors. 
\end{thm}

The proof of the Theorem above follows identically to the illustrative example below (Theorem \ref{thm-len3}). We also elaborate more in the porism after the proof of Theorem \ref{thm-len3} the general outline. Consider the right angled Artin group whose underlying graph is the line segment with 5 vertices; i.e, the graph product of $L\mathbb Z$ over the following graph:
\begin{center}
\tikz\foreach \i in {1,...,5}
    \fill (\i*72 + 90:1) coordinate (n\i) circle(2 pt)
      \ifnum \i > 4 {(n1) edge (n2)}{(n1) edge (n5)}{(n3) edge (n4)}{(n4) edge (n5)} \fi;
\end{center}
Then $M$ is generated by 5 Haar unitaries, namely $u_1$, $v_1$, $v_2,$ $v_3,$ and $u_2$ which satisfy the following relations:

$$[u_1,v_1] = [v_1,v_2] = [v_2,v_3] = [v_3,u_2] = 0.$$

\begin{thm}
\label{thm-len3}
    In the $M$ described above, $p_M(u_1,u_2) = 4.$
\end{thm}

We present now some technical lemmas regarding specific structural properties of $M$, as well as a pure real-analytic lemma. We thank Srivatsa Srinivas for helping improve our original bound of $\Theta(c^6)$ to the optimal $\Theta(c^4).$

\begin{lem}\label{lemma-lower-bound}
    Let $0 < c \leq 1.$ Define $$X = \left\{\sum_{i,j\in\mathbb Z}|\lambda_{i,j}-\lambda_{i,j+1}|^2 \middle| \lambda_{i,j}\in\mathbb C, \ \sum_{i,j\in\Z}|\lambda_{i,j}|^2 \leq 1, \ \text{and} \ \sum_{i\in\Z}|\lambda_{i,0}|^2 \geq c^2\right\}.$$
    Then $\frac{c^4}{3+18c^2}$ is a lower bound for $X.$ (Note that $X \subset \mathbb R_{\geq0}.$)
\end{lem}

\begin{proof}
    Set $v_j := (\lambda_{i,j})_{i\in\Z}$. Note that since the sum of all the squares of the $\lambda_{i,j}$ is bounded by 1, each $v_j \in\ell^2(\mathbb Z)$. Furthermore, we have the following conditions on the $v_j$: 
    \begin{align*}
        \|v_0\| \geq c \\
        \sum_{j\in\mathbb Z}\|v_j\|^2 \leq 1.
    \end{align*}
    
    We can therefore rewrite $X$ as: 
    $$\left\{\sum_{j\in \mathbb Z}\|v_j - v_{j+1}\|^2  \middle| v_j\in \ell^2(\mathbb Z), \ \sum_{j} \|v_j\|^2 \leq 1, \& \|v_0\| \geq c  \right\}.$$
    
    Set $\varepsilon^2 = \frac{c^4}{3+18c^2}$. Since $\frac{c^4}{3+18c^2}$ is an increasing function, we may assume without loss of generality that $\|v_0\| = c.$ Let $N\in\N$ be such that $(N+1)\varepsilon^2 \leq c^2 < (N+2)\ee^2$.
    
    Assume towards a contradiction that $\varepsilon^2$ is not a lower bound for $X$, in other words, assume that 
    $$\sum_{j\in \mathbb Z}\|v_j - v_{j+1}\|^2 < \varepsilon^2. $$

    Note that for $n\geq 1,$
    \begin{align*}
        \|v_n - v_0\|^2 &= \left\|\sum_{i=1}^nv_i-v_{i-1}\right\|^2\\
        &\leq \left(\sum_{i=1}^n\|v_i-v_{i-1}\|\right)^2 \\
        &\leq n\sum_{i=1}^n\|v_i-v_{i-1}\|^2 \\
        &< n\ee^2
    \end{align*}
    where for the second inequality we applied Cauchy-Schwarz. We get a similar inequality for $n\leq -1.$
    
    We therefore have that for all $|n| \leq N$ we have $\|v_n\| \geq \|v_0\| - \varepsilon\sqrt{|n|}\geq 0.$ Therefore
    \begin{align*}
        \sum_{n\in\mathbb Z} \|v_n\|^2 &\geq \sum_{n=-N}^N \|v_n\|^2 \\
        &\geq \sum_{n=-N}^N (\|v_0\| - \ee\sqrt{|n|})^2.
    \end{align*}

    Using that $\sum_{n=0}^N n = \frac{N(N+1)}{2}$ and $\sum_{n=0}^N \sqrt{n} < \frac{2}{3}(N+1)^{3/2}$ we see that 
    $$ \sum_{n=-N}^N (\|v_0\| - \ee\sqrt{|n|})^2 > (2N+1)\|v_0\|^2 + N(N+1)\varepsilon^2 - \frac{8}{3}\|v_0\|\ee (N+1)^{3/2}.$$

   Using our assumption that $(N+1)\varepsilon^2 \leq c^2 < (N+2)\ee^2$, we get that 

    $$(2N+1)\|v_0\|^2 + N(N+1)\varepsilon^2 - \frac{8}{3}\|v_0\|\ee (N+1)^{3/2} \geq \frac{1}{3}\frac{c^4}{\ee^2} - 6c^2 + 2\ee^2 > \frac{1}{3}\frac{c^4}{\ee^2} - 6c^2. $$

    By our choice of $\ee,$ we have that $\frac{1}{3}\frac{c^4}{\ee^2} - 6c^2 = 1.$

    All together, we see that
    \begin{align*}
        \sum_{n\in\mathbb Z} \|v_n\|^2 &\geq \sum_{n=-N}^N (\|v_0\| - \varepsilon\sqrt{|n|})^2 \\
        &> (2N+1)\|v_0\|^2 + N(N+1)\varepsilon^2 - \frac{8}{3}\|v_0\|\ee(N+1)^{3/2} \\
        &>\frac{1}{3}\frac{c^4}{\ee^2} - 6c^2\\
        &= 1,
    \end{align*}
    which contradicts our hypothesis that $\sum_j \|v_j\|^2 \leq 1.$

\end{proof}

\begin{remark}
    The order of the bound computed in Lemma \ref{lemma-lower-bound} is tight. Indeed, if $\|v_0\| = c$, pick $N \approx \frac{3}{2c^2}$, and for $|n|\leq N$ set $v_n = \frac{N-|n|}{N}v_0,$ and $v_n = 0$ otherwise. Then $\sum_{n}\|v_n-{v_{n+1}}\|^2 \approx 2N\frac{c^2}{N^2} \approx \frac{4c^4}{3}$.

\end{remark}

\begin{lem}\label{annoying lemma}
    Suppose $w \in \{u_1\}' \cap \cH(M^\cU)$. Then $E_{(\{v_2\}'')^\cU}(w) = 0.$
\end{lem}

\begin{proof}
    Recall that the words $$r_K(u_1,v_1,v_2,v_3,u_2) = u_1^{k_1}v_1^{k_2}v_2^{k_3}v_3^{k_4}u_2^{k_5}u_1^{k_6}v_1^{k_7}\cdots$$ SOT-densely span $M$ where $K = (k_i)_i \in \bigoplus_{i\in\mathbb N} \mathbb Z$. Therefore, by Kaplansky's density theorem, we may write $w = (w_{n})_n$ where
    $w_{n} = \sum_{i,j\in\mathbb Z} \lambda_{i,j,n}u_1^jv_2^iu_1^{-j} + y_n $ where $y_n$ is in the span of words not of the form $u_1^jv_2^iu_1^{-j}$. Additionally, we note that we can take $\lambda_{0,j,n} = 0$ for all $j,n$ since $\tau(w) = 0$. Since $u_1$ and $v_2$ are freely independent, the words of the form $u_1^jv_2^iu_1^{-j}$ with $i\neq 0$ are orthogonal. We may further assume that for all $n,$
    \begin{align}
    \label{ineq1}
        \sum_{i,j}|\lambda_{i,j,n}^2| \leq 1,
    \end{align} since $w$ is a unitary.
    
    Then $E_{(\{v_2\}'')^\cU}(w) = (E_{\{v_2\}''}(w_{n}))_n = (\sum_{i\in \mathbb Z} \lambda_{i,0,n} v_2^i)_n.$ Each $v_2^i$ is orthogonal so that 
    $$\|E_{(\{v_2\}'')^\cU}(w)\|_2^2 = \lim_{n\to\cU}\sum_{i\in \mathbb Z} |\lambda_{i,0,n}|^2. $$

    Now suppose for a contradiction that $E_{(\{v_2\}'')^\cU}(w) \neq 0.$ Then there is $\delta>0$ such that $$\lim_{n\to\cU} \sum_{i\in \mathbb Z} |\lambda_{i,0,n}|^2 = 2\delta^2 > 0.$$ In particular, there exists $S\in\cU$ such that for all $n \in S$
    \begin{align}\label{ineq2}
        \sum_{i\in \mathbb Z} |\lambda_{i,0,n}|^2 > \delta^2.
    \end{align}

    Consider the subspace $X$ of $L^2(M)$ which is 2-norm-densely spanned by the words $u_1^jv_2^iu_1^{-j}$. It is clear that $X$ and $X^\perp$ are invariant under conjugation by $u$. 
    
    Since $w$ commutes with $u_1,$ there is $T\in\cU$ such that for all $n\in T$,$\|u_1w_{n}u_1^{-1} - w_{n}\|_2^2 < \frac{\delta^4}{3+18\delta^2}$.

    However, we compute that 
    \begin{align*}
        \|u_1w_{n}u_1^{-1} - w_{n}\|_2^2 &= \|\sum_{i,j\in\mathbb Z}( \lambda_{i,j,n}u_1^jv_2^iu_1^{-j} - \lambda_{i,j,n}u_1^{j+1}v_2^iu_1^{-j-1})\|_2^2 \\
        &+\|u_1y_nu_1^{-1} - y_n\|_2^2 \\
        &\geq \|\sum_{i,j\in\mathbb Z}(\lambda_{i,j,n} - \lambda_{i,j+1,n})u_1^jv_2^iu_1^{-j}\|_2^2 \\
        &= \sum_{i,j\in\mathbb Z}|\lambda_{i,j,n} - \lambda_{i,j+1,n}|^2.
    \end{align*}

    We now have that for all $n \in S\cap T,$  $\sum_{i,j}|\lambda_{i,j,n}^2| \leq 1$ by (\ref{ineq1}) and $\sum_i |\lambda_{i,0,n}|^2 \geq \delta^2$ by (\ref{ineq2}). By the previous lemma, we get that $$\|u_1w_{n}u_1^{-1} - w_{n}\|_2^2\geq \sum_{i,j\in\mathbb Z}|\lambda_{i,j,n} - \lambda_{i,j+1,n}|^2 \geq \frac{\delta^4}{3+18\delta^2},$$ a contradiction.
    
\end{proof}

\begin{lem}
    Suppose $w_i \in \{u_i\}' \cap \cH(M^\cU)$ for $i=1,2$. Then $w_1$ and $w_2$ are freely independent. In particular, they do not commute.
\end{lem}

\begin{proof}
    From the previous lemma we know that $E_{(\{v_2\}'')^\cU}(w_i) = 0$. Write $M$ as an amalgamated free product over $\{v_2\}''$ in the following way: $M = \{u_1,v_1,v_2\}'' *_{\{v_2\}''} \{v_2,v_3,u_2\}''$. Now apply Theorem \ref{HI independence} to conclude.
\end{proof}

\begin{proof}[Proof of Theorem \ref{thm-len3}]

    Suppose there were $w_1,w_2 \in \cH(M^\cU)$ such that $$[u_1,w_1] = [w_1,w_2] = [w_2,u_2] = 0.$$

    By the previous lemma, if $[u_1,w_1] = [w_2,u_2] = 0$ then $w_1$ and $w_2$ are freely independent, a contradiction.

    So the only paths between $u_1$ and $u_2$ must be length 4 or longer.
    
\end{proof}

\begin{proof}[Proof of Theorem \ref{thm-graph-radius}]

If $G = (V,E)$ is a graph and $M_v$ is a tracial diffuse von Neumann algebra for each $v\in V,$ then we can consider the graph product $M$ of the $M_v$ over $G.$ Suppose $v_1,v_2 \in V$ are vertices distance $d\geq 4$ apart. We use essentially the same argument as Theorem \ref{thm-len3} but one must amalgamate over a subgraph that separates $v_1$ and $v_2.$  Namely take $M_1$ to be the von Neumann algebra of  the subgraph containing the vertices that are  of distance at most 2 from $v_1$; $B$  to be the von Neumann algebra of  the subgraph containing the vertices that are  of distance  equal to 2 from $v_1$; $M_2$ to be the von Neumann algebra of  the subgraph containing the vertices that are  of distance greater than or equal to 2 from $v_1$. Note that there are no edges between the vertices of distance 0 or 1 from $v_1$ and vertices of distance 3 or more from $v_1$, so that we then have the decomposition $M\cong M_1*_B M_2$. An adaptation of the argument in Lemma \ref{annoying lemma} with words in multiple letters from $B$ as opposed to just a single letter, gets us the computation  that if $w \in \{v_1\}' \cap \cH(M^\cU)$, then $E_{{B}^\cU}(w) = 0.$   Then  applying Theorem \ref{HI independence}, as before we get that there is a pair of orthogonal unitaries $v_1,v_2\in M^\mathcal{U}$ such that there is no sequentially commuting length $3$  path of Haar unitaries between them, which implies that $M^\mathcal{U}$ is never isomorphic to $\mathcal{S}_2(L(SL_3(\mathbb{Z})))^{\mathcal{V}}$ for any ultrafilter $\mathcal{V}$ as required. 
    
\end{proof}

\bibliographystyle{amsalpha}
\bibliography{inneramen}

\end{document}